\documentclass[11pt, oneside]{amsart}   	
\usepackage{geometry}                		
\usepackage{graphicx}				
								
\usepackage{amssymb}
\usepackage{bm}

\newcommand\Kl{\mathrm{Kl}}
\newcommand\sumast{\mathop{\sum\nolimits^{\ast}}}
\newcommand\OO{\mathop{\mathrm{O}}}
\newcommand\ord{\mathop{\mathrm{ord}}\nolimits}
\DeclareMathOperator*{\argmin}{arg\,min}
\newcommand\exmid{\mathrel{\|}}

\makeatletter
\newcommand{\tpmod}[1]{\mkern 8mu({\operator@font mod}\mkern 6mu#1)}
\makeatother

\newcommand\bma{\bm{a}}
\newcommand\bme{\bm{\epsilon}}

\theoremstyle{plain}
\newtheorem{theorem}{Theorem}
\newtheorem{proposition}[theorem]{Proposition}
\newtheorem{lemma}[theorem]{Lemma}
\newtheorem{definition}{Definition}

\theoremstyle{remark}
\newtheorem{remark}{Remark}
\newtheorem{example}{Example}

\numberwithin{equation}{section}
\numberwithin{theorem}{section}

\usepackage{xcolor}
\definecolor{couleur_cite}{rgb}{0.05,.4,0.05}
\definecolor{couleur_link}{rgb}{0.05,0.05,0.4}
\definecolor{couleur_url}{rgb}{0.5,0,0}
\usepackage[hyperfootnotes=false]{hyperref}
\hypersetup{hypertexnames=false,colorlinks=true,citecolor=couleur_cite,linkcolor=couleur_link,urlcolor=couleur_url,
pdfstartview=FitH, pdfauthor=Djordje Milicevic and Catherine Robinson and Chloe Shupe, pdftitle=Sums of products of Kloosterman sums to prime power moduli}

\title[Sums of products of Kloosterman sums to prime power moduli]{Sums of products of Kloosterman sums\\ to prime power moduli}
\author{Djordje Mili\'cevi\'c}
\address{Bryn Mawr College, Department of Mathematics, 101 North Merion Avenue, Bryn Mawr, PA 19010, USA}
\email{dmilicevic@brynmawr.edu}

\author{Catherine Robinson}
\address{CHA Consulting, Inc., Philadelphia, PA}
\email{catierobinhood@gmail.com}

\author{Chloe Shupe}
\address{Rice University, Department of Mathematics, 6100 Main Street, Houston, TX 77005, USA}
\email{cs235@rice.edu}

\thanks{Research supported in part by the Simons Foundation Award MPS-TSM-00008085 and the National Science Foundation Grant DMS-1903301 (D.M.).}

\begin{document}

\begin{abstract}
We prove new bounds on complete sums of products of $k$ additively shifted Kloosterman sums to odd high prime power moduli $q=p^n$, which feature substantially stronger power savings (about $q^{-1/\lceil k/2\rceil}$ in generic configurations) and a novel quantification of the alignment among the shifts. We prove our bounds by developing a method to estimate complete exponential sums with a broad class of phases well-controlled by a specified number of initial terms in a property reflecting $p$-adic differentiability.
\end{abstract}

\maketitle

\section{Introduction}

\subsection{Cancellation paradigm}
Sums of products of complete exponential sums to a large modulus $q$ are a major player frequently encountered across analytic number theory. Typically they arise as follows: following an application of, say, Poisson summation formula or an automorphic tool such as the Vorono{\u\i} or Kuznetsov trace formula, complete exponential sums (such as Gauss or Kloosterman sums) with varying parameters arise; then, one applies the Cauchy--Schwarz or H\"older's inequality and encounters in the off-diagonal terms the complete sums of products of these complete exponential sums, such as \eqref{basic-sum} below, hoping to be able to profit from a substantial cancellation in such sums, including (ideally) the \emph{square-root cancellation} in generic configurations of parameters. Such a situation arises, for example, in studying moments of twisted $L$-functions \cite{BlomerMilicevic2015a,KowalskiMichelSawin2017,MilicevicQinWu2025}, algebraic twists of modular forms~\cite{FouvryKowalskiMichel2015a}, shapes of exponential sum paths \cite{KowalskiSawin2016,MilicevicZhang2023,DellMilicevic2025}, and more; in the context of sums of products of character values, it underlies Burgess' classical bound for short character sums~\cite{Burgess1963}. We refer the reader to \cite{FouvryKowalskiMichel2015} for a wonderful general account of sums of products in analytic number theory.

In particular, for $(a,q)=1$, consider the normalized Kloosterman sum
\begin{equation}
\label{Kl2}
\Kl_2(a;q)=\frac1{\sqrt{q}}\sumast_{x\bmod q}e\Big(\frac{x+a\bar{x}}q\Big),
\end{equation}
which satisfies the bound $|\Kl_2(a;q)|\leqslant 2^{\omega(q)}=q^{o(1)}$ by Weil's bound for prime $q$ \cite{Weil1948}, by a $p$-adic stationary phase argument for a prime power $q=p^n$ (see, for example, \cite{BlomerMilicevic2015}), and then by twisted multiplicativity in general. Then one often encounters the additively shifted sum of products (for $\bm{a}\in(\mathbb{Z}/q\mathbb{Z})^k$)
\begin{equation}
\label{basic-sum}
S(\bm{a};q)=\sum_{\substack{x\bmod q\\(x+a_i,q)=1}}\Kl_2(x+a_1;q)\Kl_2(x+a_2;q)\cdots\Kl_2(x+a_k;q),
\end{equation}
which is the central subject of the present paper.
For $a_i\not\equiv a_j\bmod q$ and $q$ squarefree, the sums $\Kl_2(x+a_i;q)$ and $\Kl_2(x+a_j;q)$ are uncorrelated in a precise sense, and the same is true in general unless the shifts $a_i$ and $a_j$ align to a high divisor of $q$. One might thus optimistically expect that the sum \eqref{basic-sum} would exhibit significant (perhaps square-root) cancellation beyond the trivial bound $S(\bm{a};q)\ll q^{1+o(1)}$ unless there is a clear reason for the failure of such cancellation in terms of alignment (perhaps pairwise alignment) of the additive shifts. Indeed, for a prime modulus $q$, tools from algebraic geometry lead to the now-classical result that
\begin{equation}
\label{sqroot-prime}
S(\bm{a};q)\ll_{k,\epsilon}q^{1/2+\epsilon}\qquad
\begin{aligned}&\text{unless}\quad k=2m\quad\text{and}\\
&\exists\sigma\in S_{2m}\text{ s.t.\ }a_{\sigma(i)}=a_{\sigma(m+i)}\,(1\leqslant i\leqslant m).\end{aligned}
\end{equation}
For this bound and its more general versions, see Fouvry--Michel--Rivat--S\'ark\"ozy~\cite[Lemma 2.1]{FouvryMichelRivatSarkozy2004}, Fouvry--Ganguly--Kowalski--Michel~\cite[Proposition 3.2]{FouvryGangulyKowalskiMichel2014}), and the references in \cite[\S1]{FouvryKowalskiMichel2015}.

\subsection{Prime power moduli and singular phases}
For applications, one is naturally interested in bounds on $S(\bm{a};q)$ that apply to arbitrary moduli $q$. By multiplicativity, the study of complete exponential sums such as \eqref{Kl2} and \eqref{basic-sum} essentially boils down to the prime power case $q=p^n$, with the case $n=1$ answered by \eqref{sqroot-prime}.

For $n\geqslant 2$, algebro-geometric tools are not directly applicable, and, for large $n$, sums of the form
\begin{equation}
\label{Sf}
S_f:=\sum_{x\bmod q}e\Big(\frac{f(x)}q\Big),
\end{equation}
where the function $f$ defined on some Zariski open subset of $\mathbb{Z}/q\mathbb{Z}$ is a rational function or, more generally, a $p$-adically smooth function in the sense that $f(x_0+p^{\kappa}t)$ ($t\in\mathbb{Z}$) may be expanded into a Taylor polynomial of suitably high degree, are instead best understood using the $p$-adic method of stationary phase such as our Lemma~\ref{statphase-lemma}. Under suitable integrality conditions, this shows that the sum \eqref{Sf} may be restricted to $f'(x)\equiv 0\bmod p^{\lfloor n/2\rfloor}$ (namely, that the contributions of $x$ not satisfying this condition fully cancel out). The congruence $f'(x)\equiv 0\bmod p$ ($p$-adically, the coarse approximation of this congruence) admits $\OO(\min(p,\deg(f\bmod p)))$ solutions, and if, additionally $f''(x)\not\equiv 0\bmod p$ (the situation termed a non-singular stationary point in analogy with the corresponding Archimedean situation of Newton's method in numerical analysis), then Hensel's Lemma can be used to show that each of these solutions lifts to a unique solution modulo $p^{\lfloor n/2\rfloor}$, whence $S_f\ll p^{n/2+\OO(1)}$. This is indeed what happens in some interesting situations, like the very evaluation of \eqref{Kl2} as in Lemma~\ref{kloost-eval}. For $q=p^2$, Fouvry and Ganguly proved square-root strength bounds on products of arbitrarily many Kloosterman sums with fixed $\mathrm{GL}_2$ shifts~\cite{FouvryGanguly2020}, a beautiful result which, however, excludes possible alignments in shifts $a_i$.

Alas, singular stationary points can and do occur in arithmetically interesting sums \eqref{Sf} to prime power moduli; this is famously the cause for the restriction to cube-free moduli $q$ in Burgess' method and its applications such as \cite{Heath-Brown1978a} and the major additional ingredient in the removal of the same restriction in the subconvexity problem for Dirichlet $L$-functions~\cite{PetrowYoung2023}. No direct analogue of the uniform bound \eqref{sqroot-prime} for the sum of products of Kloosterman sums modulo $q=p^n$ ($n\geqslant 2$) is available in the literature for $k>2$, and indeed, as we shall see, such a statement provably fails in many situations where the entries of $\bm{a}$ do not come in pairs of equals as in \eqref{sqroot-prime}, the only situation where it is indeed clear that one cannot expect significant cancellation. This problem was first considered by Mili\'cevi\'c and Zhang, whose results \cite[Theorems 3 and 4]{MilicevicZhang2023} imply that, for every $k\in\mathbb{N}$, there exist constants $\delta_i=\delta_i(k)>0$ ($i=1,2$) such that
\begin{equation}
\label{MZ-estimate}
S(\bm{a};q)\ll_kq^{1-\delta_1}\qquad
\text{unless}\quad a_i\equiv a_j\bmod q^{\delta_2}\text{ for some }1\leqslant i\neq j\leqslant k.
\end{equation}
The proof of \eqref{MZ-estimate} deals with the possibility of highly (and deeply) singular points in the sense that the relevant exponential sum $S_f$ as in \eqref{Sf} can, by a recursive argument, be restricted to points $x$ at which the derivatives of the phase $f^{(j)}$ satisfy simultaneously
\[ f^{(j)}(x)\equiv 0\bmod{q^{\eta_j}} \]
for many different values of $j$ (and constants $\eta_j>0$ depending only on $k$), and it is only by the time such conditions are captured for $j\leqslant k$ that one concludes that this can only happen when at least two of the shifts $a_i$ agree as in \eqref{MZ-estimate}. This shows that the sum $S(\bm{a};q)$ indeed exhibits some power cancellation unless there is at least some agreement among the shifts $\bm{a}$, but the method only yields $\delta_i\asymp 1/2^k$, and it is easy to believe that neither this power savings nor the mere slight alignment of two (among many) shifts in \eqref{MZ-estimate} give a full description of the true size of $S(\bm{a};q)$. One desires --- for applications as much as out of the raw need to know --- an upper bound that is, whether of square-root quality or not, demonstrably the best possible. Moreover, it is not even clear from \cite{MilicevicZhang2023} that these highly and deeply singular points that cause so much trouble can in fact occur.

\subsection{Measures of alignment and the main result}
In this paper, we seek to get closer to the true size of $S(\bm{a};q)$. We begin with defining a different measure of alignment between two tuples.

\begin{definition}
\label{alignment-partial}
For every $\bm{a}=(a_1,\dots,a_k),\bm{b}=(b_1,\dots,b_k)\in(\mathbb{Z}/p^n\mathbb{Z})^k$, let
\begin{gather}
\notag \delta_j(\bm{a},\bm{b})=\Big(p^n,\sum_{i=1}^k(b_i^j-a_i^j)\Big),\quad \Delta_j(\bm{a},\bm{b})=\ord_p\delta_j(\bm{a},\bm{b}),\\
\label{Delta-Def}
\Delta(\bm{a},\bm{b})=\min_{1\leqslant j\leqslant k}\Delta_j(\bm{a},\bm{b}).
\end{gather}
\end{definition}
It is not hard to see (see Lemma~\ref{recurrence-lemma-plus-p}) that $\Delta=\Delta(\bm{a},\bm{b})$ is (for sufficiently large primes $p\geqslant p_0(k)$ initially) the highest exponent up to $n$ such that \emph{all} corresponding power sums of $\bm{a}$ and $\bm{b}$ agree. In light of Newton's formulas for symmetric polynomials (see \S\ref{alignment-p-sec} for details), this is (still for $p\geqslant p_0(k)$) precisely the highest power such that the congruence of polynomials
\begin{equation}
\label{alignment-poly}
(X-a_1)(X-a_2)\cdots(X-a_k)\equiv (X-b_1)(X-b_2)\cdots(X-b_k)\pmod{p^{\Delta}}
\end{equation}
holds. For smaller primes, the same claims are true up to factors of size $p^{\rho_k}$, with $\rho_k$ depending on $k$ only.

Now, we may also consider the more immediate measure of alignment between tuples $\bm{a},\bm{b}\in(\mathbb{Z}/p^n\mathbb{Z})^k$ given by
\[ \Delta^0(\bm{a},\bm{b})=\max\big\{0\leqslant m\leqslant n:\bm{b}\equiv\sigma(\bm{a})\pmod{p^m}\text{ for some }\sigma\in S_k\big\}. \]
Now, from \eqref{alignment-poly} (and \eqref{Delta-Def}) it is clear that $\Delta\geqslant\Delta^0=\Delta^0(\bm{a},\bm{b})$, and, for $\Delta=1$, the condition \eqref{alignment-poly} is indeed equivalent to $\bm{b}\equiv\sigma(\bm{a})\pmod p$ for some $\sigma\in S_k$. Thus, for $p\geqslant p_0(k)$, we have that $\Delta>0$ if and only if $\Delta^0>0$ (the same also being true for smaller primes up to factors $p^{\rho_k}$, with $\rho_k$ depending on $k$ only), and, moreover, in the case of prime modulus $p\geqslant p_0(k)$, these two measures of alignment agree: $\Delta=\Delta^0$. With prime power moduli, however, \eqref{alignment-poly} can in general hold to a substantially higher power than $\Delta^0$: if $p\equiv 1\pmod k$, $n=km$, and $\mu_k$ is a primitive $k$th root in $(\mathbb{Z}/p^n\mathbb{Z})^{\times}$, then
\[ (X-\mu_kp^m)(X-\mu_k^2p^m)\cdots(X-\mu_k^{k-1}p^m)(X-p^m)\equiv X^k\pmod{p^n}, \]
and $\bm{a}=(\mu_kp^m,\mu_k^2p^m,\dots,\mu_k^{k-1}p^m,p^m)$ and $\bm{b}=\bm{0}$ satisfy $\Delta(\bm{a},\bm{b})=n$ despite the two $k$-tuples only agreeing modulo $p^m=p^{n/k}$ (that is, $\Delta^0(\bm{a},\bm{b})=n/k$). We will show in Lemma~\ref{1-over-k} that, for $p\geqslant p_0(k)$ and $\Delta(\bm{a},\bm{b})>0$,
\[ \frac1k\leqslant\frac{\Delta^0(\bm{a},\bm{b})}{\Delta(\bm{a},\bm{b})}\leqslant 1. \]

Definition~\ref{alignment-partial} allows us to quantify the degree to which the $k$-tuple $\bm{a}\in(\mathbb{Z}/p^n\mathbb{Z})^k$ appearing in \eqref{basic-sum} is non-generic for the purposes of estimating the sum $S(\bm{a};q)$.

\begin{definition}
\label{Deltaast-maindef}
Given $\bm{a}\in(\mathbb{Z}/p^n\mathbb{Z})^k$, let $T=T(\bm{a})=\{[a_i\bmod p]:1\leqslant i\leqslant k\}$ be the set of attained residues modulo $p$, and reindex the variables in the multiset $[a_1,a_2,\dots,a_k]$ according to the disjoint decomposition
\[ [a_1,a_2,\dots,a_k]=\bigsqcup_{t\in T}[a_{t1},a_{t2},\dots,a_{tk_t}],\quad a_{ti}\equiv t\pmod p\quad (1\leqslant i\leqslant k_t). \]
For every $t\in T$, denote for every subset $I\subseteq[k_t]=\{1,2,\dots,k_t\}$, $\bm{a}_t^I=(a_{ti})_{i\in I}$, and
\begin{equation}
\label{Delta-ta-def}
\Delta_t(\bm{a})=\max_{\substack{[k_t]=I\sqcup I'\\|I|=|I'|=k_t/2}}\Delta(\bm{a}_t^I,\bm{a}_t^{I'}),
\end{equation}
where $\Delta(\bm{a}_t^I,\bm{a}_t^{I'})$ is as in Definition~\ref{alignment-partial}, the maximum is taken over all decompositions of $[k_t]$ into two subsets of equal size $k_t/2$, and $\Delta_t(\bm{a})=0$ if $k_t$ is odd. Finally, define
\[ \Delta^{\ast}(\bm{a})=\min_{t\in T}\Delta_t(\bm{a}). \]
\end{definition}

In particular, we see that $\Delta^{\ast}(\bm{a})=0$ unless $k=2\kappa$ and there exists a permutation $\sigma\in S_k$ such that
\[ a_{\sigma(i)}\equiv a_{\sigma(\kappa+i)}\pmod p\quad(1\leqslant i\leqslant\kappa). \]
Further, if we denote analogously by
\[ \Delta^{0\ast}(\bm{a})=\max\big\{0\leqslant m\leqslant n:a_{\sigma(i)}\equiv a_{\sigma(\kappa+i)}\pmod{p^m}\,(1\leqslant i\leqslant\kappa)\text{ for some }\sigma\in S_k\big\}, \]
the highest exponent $m$ such that there is pairwise alignment of shifts in $\bm{a}$ modulo $p^m$, then $\Delta^{0\ast}(\bm{a})=\min_{t\in T}\max_{[k_t]=I\sqcup I',|I|=|I'|=k_t/2}\Delta^0(\bm{a}_t^I,\bm{a}_t^{I'})$, and invoking Lemma~\ref{1-over-k} we have that for $p\geqslant p_0(k)$ and $\Delta^{\ast}(\bm{a})>0$,
\[ \min_{t\in T}\frac1{k_t/2}\leqslant\frac{\Delta^{0\ast}(\bm{a})}{\Delta^{\ast}(\bm{a})}\leqslant 1. \]

We are now ready to state our main result.

\begin{theorem}
\label{main-theorem}
For every odd prime $p$, $q=p^n$, $\bm{a}\in(\mathbb{Z}/p^n\mathbb{Z})^k$, we have
\begin{equation}
\label{main-estimate}
S(\bm{a};q)\ll_kp^{n-{}{\textstyle\frac{n-\Delta^{\ast}(\bm{a})-1}{\lceil k/2\rceil}}+1},
\end{equation}
where $0\leqslant\Delta^{\ast}(\bm{a})\leqslant n$ is as in Definition~\ref{Deltaast-maindef}.
\end{theorem}

\begin{remark}
We in fact develop a general method to estimate complete exponential sums with phases in a class $\mathcal{F}_k(\mathbb{Z}/p^n\mathbb{Z})$ which consists of functions which allow for sufficient application of the method of stationary phase and which are well-controlled by a specified number of initial terms in a property reflecting $p$-adic differentiability. We define this class precisely in Definition~\ref{Fk-def} and then prove in Proposition~\ref{Fk-final-estimate}, which is of independent interest, a general estimate of the form
\begin{equation}
\label{sum-general-intro}
\sum_{x\in A_F}e\Big(\frac{F(x)}{p^n}\Big)\ll_kp^{n-\frac{n-\Delta_k(F)-1}{k}+1},
\end{equation}
with notations as in Definition~\ref{Fk-def}. We obtain the bound \eqref{main-estimate} in Theorem~\ref{main-theorem} by showing (which is not at all straightforward) that phases arising in $S(\bm{a};q)$ belong to the class $\mathcal{F}_{\lceil k/2\rceil}(\mathbb{Z}/p^n\mathbb{Z})$ and explicating the corresponding $\Delta_{\lceil k/2\rceil}(F)$ in terms of our measure of alignment $\Delta^{\ast}(\bm{a})$.
\end{remark}

\begin{remark}
The bound \eqref{main-estimate} marks a significant improvement over the previously known best result \eqref{MZ-estimate}. For all we know, it may well be that, in many cases, it can be further improved. However, one has to be careful with expectations in full generality. In particular, if $\Delta^{\ast}(\bm{a})=n$, then no structural cancellation can be expected in $S(\bm{a};q)$. For example, for $p$ and $\mu_k$ as in the discussion under \eqref{alignment-poly} and $\bm{a}=(\mu_kp^m,\mu_k^2p^m,\dots,p^m,0,0,\dots,0,0)\in(\mathbb{Z}/p^n\mathbb{Z})^{2k}$, we will see in Example~\ref{mu-example} in section~\ref{proof-main-theorem} that the sum
\begin{equation}
\label{mu-example-eq}
S_k(\bm{a};q)=\!\!\sumast_{x\bmod q}\Kl_2(x+\mu_kp^m;q)\Kl_2(x+\mu_k^2p^m;q)\cdots\Kl_2(x+\mu_k^kp^m;q)\Kl_2(x,q)^k\asymp q
\end{equation}
exhibits \emph{no} cancellation whatsoever, despite the pairwise agreement being only modulo $p^{\Delta^{0\ast}(\bm{a})}=q^{1/k}$.
\end{remark}

\begin{remark}
The results of \cite{MilicevicZhang2023} have been applied to the distribution of square-free integers in arithmetic progressions to smooth~\cite{Mangerel2021} and prime power moduli~\cite{ZhongZhang2026}. Theorem~\ref{main-theorem} should lead to the corresponding improvements in the key propositions \cite[Proposition 4.8]{Mangerel2021} and \cite[Theorem 1.3]{ZhongZhang2026} (which may in turn improve other results in these two papers).
\end{remark}

\begin{remark}
In the elementary case $k=2$ in \eqref{main-estimate}, it is easy to see that the sum $S(\bm{a};q)$ (which coincidentally agrees with the Ramanujan sum $c_q(a_1-a_2)$) vanishes unless $\Delta^{\ast}(\bm{a})\geqslant n-1$.
\end{remark}

\subsection{Organization of the paper}
In section~\ref{preliminaries-section}, we first collect preliminaries on the so-called $p$-adic method of stationary phase and the square roots and Kloosterman sums to odd prime power moduli. We also collect some important consequences of the theory of symmetric polynomials and define the class $\mathcal{F}_k(\mathbb{Z}/p^n\mathbb{Z})$ to which our method for estimating exponential sums of the form \eqref{sum-general-intro} applies. In section~\ref{alignment-section}, we study systems of congruences arising from the condition that the first $k$ terms in phases arising in sums $S(\bm{a};q)$ vanish to a certain power $p^r$ and relate this to the quantities from which our measure of alignment $\Delta^{\ast}(\bm{a})$ is built. In section~\ref{stratification-section}, we execute a sensitive recursive lifting and stratification argument (a generalization of the Hensel's lemma and stationary phase argument) through which we split the domain of a phase $F\in\mathcal{F}_k(\mathbb{Z}/p^n\mathbb{Z})$ into large portions over which the terms $e(F(x)/p^n)$ exhibit full cancellation and $\OO_k(1)$ neighborhoods over each of which the phase is essentially non-oscillating. Finally, in section~\ref{proof-main-theorem}, we combine all the ingredients to prove in Proposition~\ref{Fk-final-estimate} our bound on the general exponential sums as in \eqref{sum-general-intro}, show that the phases arising in $S(\bm{a};q)$ belong to the class $\mathcal{F}_{\lceil k/2\rceil}(\mathbb{Z}/p^n\mathbb{Z})$, and finally prove our main Theorem~\ref{main-theorem}.

\subsection{Notation}
As is customary in analytic number theory, we write $e(z)=e^{2\pi iz}$, $f\ll g$ or $f=\OO(g)$ to denote $|f|\leqslant Cg$ for some constant $C>0$ which may vary from line to line but is otherwise absolute unless otherwise indicated by a subscript, and $f\asymp g$ if $f\ll g$ and $g\ll f$. We also write $p^{\kappa}\exmid n$ if $p^{\kappa}\mid n$ but $p^{\kappa+1}\nmid n$. We use the notations $A\sqcup B$ and $\bigsqcup_{i\in I}A_i$ to denote pairwise disjoint unions of sets. We also use the standard combinatorial notations $(a)_k=\prod_{j=0}^{k-1}(a-j)$ and $[j]=\{1,2,\dots,j\}$, as well as $x^{+}=\max(x,0)$ for $x\in\mathbb{R}$.

\section{Preliminaries and first steps}
\label{preliminaries-section}

\subsection{Square roots modulo \texorpdfstring{$p^n$}{p to n}}
\label{sqroot-section}

In this section, we describe the construction and properties of square roots modulo high powers of a fixed odd prime $p$, following closely the exposition in \cite[\S2.2]{MilicevicZhang2023}.

Kloosterman sums $\Kl_{p^n}(a,b)$ to proper prime power moduli are explicitly evaluated (see Lemma~\ref{kloost-eval}) in terms of exponentials with phases that are solutions of congruences of the form $x^2\equiv ab\pmod{p^r}$, which are of course precisely the square roots modulo $p^r$. These may be productively thought of as restrictions of the branches of the $p$-adic square roots, and their appearance is analogous to the exponentials with square root phases appearing in the asymptotics of $J$-Bessel functions, the archimedean analogue of Kloosterman sums. To simplify the exposition, we avoid the $p$-adic language from now on but refer to \cite[\S2.4]{BlomerMilicevic2015} for more details and a fully $p$-adic perspective.

For every $r\in(\mathbb{Z}/p\mathbb{Z})^{\times 2}$, there are exactly two classes $s\in\mathbb{Z}/p\mathbb{Z}$ such that $s^2=r$. We fix once and for all a choice function $s:(\mathbb{Z}/p\mathbb{Z})^{\times 2}\to(\mathbb{Z}/p\mathbb{Z})^{\times}$ such that $s(r)^2=r$ for every $r\in(\mathbb{Z}/p\mathbb{Z})^{\times 2}$ (any of the $2^{(p-1)/2}$ such choices will do). For every $x\in(\mathbb{Z}/p^n\mathbb{Z})^{\times 2}$, by Hensel's lemma
there exists a unique $u\in(\mathbb{Z}/p^n\mathbb{Z})^{\times}$ such that $u^2\equiv x\pmod{p^n}$ and $u\equiv s(x)\pmod p$; this gives rise to a square root $u:(\mathbb{Z}/p^n\mathbb{Z})^{\times 2}\to(\mathbb{Z}/p^n\mathbb{Z})^{\times}$, which we also denote by $u=u_{1/2}(x)=x_{1/2}$ and write $x_{1/2}^k=(x_{1/2})^k$ for $k\in\mathbb{Z}$.

The square-root on $(\mathbb{Z}/p^n\mathbb{Z})^{\times\, 2}$ thus defined satisfies the following property for $x\in(\mathbb{Z}/p^n\mathbb{Z})^{\times\, 2}$, $t\in\mathbb{Z}/p^n\mathbb{Z}$, $\kappa\geqslant 1$, and $m\in\mathbb{Z}$, $j\in\mathbb{N}$:
\begin{equation}
\label{sqroot-diff}
(x+p^{\kappa}t)_{1/2}^m\equiv\sum_{i=0}^j\binom{m/2}{i}x_{1/2}^{m-2i}p^{i\kappa}t^i\pmod{p^{\min((j+1)\kappa,n)}}.
\end{equation}
where we note that all coefficients are in $\mathbb{Z}/p^n\mathbb{Z}$. The congruence (which is obtained by multiplying $x_{1/2}^m$ times the formal series expansion of $(1+p^{\kappa}x^{-1}t)^{m/2}$) is easily verified by confirming (using the Binomial Theorem) that the squares of both sides agree modulo $p^{\min((j+1)\kappa,n)}$ and that both sides agree modulo $p$; cf.~\cite[(2.6)]{BlomerMilicevic2015}.

With small variations, this construction can be extended to the case $p=2$; see, for example, \cite[\S 3.1]{MilicevicQinWu2025}. We would expect all results of the present paper to admit similar adaptations, but refrain from this to lessen the notational overload.

\subsection{Method of stationary phase}
In this section, we collect from \cite{MilicevicZhang2023} facts about the so-called $p$-adic method of stationary phase, a powerful tool in the study of complete exponential sums modulo prime powers analogous to the classical method of stationary phase for oscillatory exponential integrals. We phrase the results with emphasis on differentiability-like properties \eqref{diffble-eq1} and \eqref{diffble-eq2} but without invoking $p$-adic language. We refer to \cite{MilicevicZhang2023} for details and proofs and to \cite[Lemmata 12.2 and 12.3]{IwaniecKowalski2004}, \cite[Lemma~7]{BlomerMilicevic2015} for other variations of the method.

\begin{lemma}[Method of stationary phase, {\cite[Lemma~1]{MilicevicZhang2023}}]
\label{statphase-lemma}
Let $1\leqslant\kappa_0\leqslant n$, and let $T\subseteq\mathbb{Z}/p^n\mathbb{Z}$ be a set invariant under translations by $p^{\kappa_0}\mathbb{Z}/p^n\mathbb{Z}$.
\begin{enumerate}
\item\label{MSP-claim1} Suppose that functions $f,f_1:T\to\mathbb{Z}/p^n\mathbb{Z}$ satisfy
\begin{equation}
\label{diffble-eq1}
f(x+p^{\kappa}t)\equiv f(x)+f_1(x)\cdot p^{\kappa}t\pmod{p^{2\kappa}}
\end{equation}
for all $x\in T$, $t\in\mathbb{Z}/p^n\mathbb{Z}$, and $\kappa\geqslant\kappa_0$. Then, for every $\max(\kappa_0,n/2)\leqslant\kappa\leqslant n$,
the set $\{x\in T:f_1(x)\equiv 0\bmod p^{n-\kappa}\}$ is invariant under translations by $p^{\kappa}\mathbb{Z}/p^n\mathbb{Z}$, and
\[ \sum_{x\in T}e\left(\frac{f(x)}{p^n}\right)=\sum_{\substack{x\in T\\f_1(x)\equiv 0\bmod{p^{n-\kappa}}}}e\left(\frac{f(x)}{p^n}\right)=p^{n-\kappa}\sum_{\substack{x\in T/p^{\kappa}\mathbb{Z}\\f_1(x)\equiv 0\bmod{p^{n-\kappa}}}}e\left(\frac{f(x)}{p^n}\right). \]
\item Suppose that $n\geqslant 2$ and functions $f,f_1:T\to\mathbb{Z}/p^n\mathbb{Z}$, $f_2:T\to(\mathbb{Z}/p^n\mathbb{Z})^{\times}$  satisfy
\begin{equation}
\label{diffble-eq2}
f(x+p^{\kappa}t)\equiv f(x)+f_1(x)\cdot p^{\kappa}t+\bar{2}f_2(x)\cdot p^{2\kappa}t^2\pmod{p^{2\kappa+1}}
\end{equation}
for all $x\in T$, $t\in\mathbb{Z}/p^n\mathbb{Z}$, and $\kappa\geqslant\kappa_0$, and suppose that $\kappa_0\leqslant\lfloor n/2\rfloor$. Then,
writing $n=2\kappa+\rho$ with $\rho\in\{0,1\}$,
\[ \sum_{x\in T}e\left(\frac{f(x)}{p^n}\right)=p^{n/2}
\sum_{\substack{x_0\in T/p^{\kappa}\mathbb{Z}\\ f_1(x_0)\equiv 0\bmod p^{\kappa}}}
\varepsilon\big(2f_2(x_0),p^{\rho}\big)e\bigg(\frac{f(x_0)-\overline{2f_2(x_0)}{f_1(x_0)^2}}{p^n}\bigg), \]
where
$\varepsilon(\cdot,1)=1$ and $\varepsilon(\cdot,p)=(\cdot/p)i^{(\iota-1)/2}$ for $p\equiv \iota\bmod 4$, $\iota\in\{1,3\}$.
\end{enumerate}
\end{lemma}

An important application of Lemma~\ref{statphase-lemma} and the construction in \S\ref{sqroot-section} is the following classical evaluation of Kloosterman sums to odd prime power moduli.

\begin{lemma}[{\cite[Lemma~2]{MilicevicZhang2023}}]
\label{kloost-eval}
Let $n\geqslant 2$ and $a\in\mathbb{Z}/p^n\mathbb{Z}$, $b\in(\mathbb{Z}/p^n\mathbb{Z})^{\times}$. Then,  $\Kl_{p^n}(a,b)=0$ if $ab\not\in (\mathbb Z/p^n\mathbb Z)^{\times 2}$. Otherwise, if $ab\in (\mathbb Z/p^n\mathbb Z)^{\times 2}$,
\[ \Kl_{p^n}(a,b)=2\left(\frac{(ab)_{1/2}}{p^n}\right)\mathrm{Re}\big[\varepsilon_{p^n}e_{p^n}\big(2(ab)_{1/2}\big)\big], \]
where $(\cdot)_{1/2}$ refers to the square root introduced in \S\ref{sqroot-section}, $(\cdot/p^n)$ is the Jacobi symbol, and
\[\varepsilon_{p^n}=\begin{cases}
1,&\text{if $2\mid n$ or $p\equiv 1\bmod 4$},\\
i, &\text{if $2\nmid n$ and $p\equiv 3\bmod 4$}.
\end{cases}\]
\end{lemma}

\subsection{Symmetric polynomials and the recurrence argument}
\label{recurrence-sec}

In this subsection, we collect some important consequences of the theory of symmetric polynomials. For an $\bm{x}=(x_1,x_2,\dots,x_k)\in\mathbb{Z}^k$, consider its power sums and elementary symmetric polynomials
\begin{equation}
\label{elementary-symmetric}
\sigma_m(\bm{x})=\sum_{j=1}^kx_j^m,\quad e_m(\bm{x})=\sum_{1\leqslant j_1<j_2<\dots<j_m\leqslant k}x_{j_1}x_{j_2}\dots x_{j_m}.
\end{equation}
For brevity of notation, in this subsection we sometimes suppress $\bm{x}$ from notation and simply write $\sigma_m=\sigma_m(\bm{x})$ and $e_m=e_m(\bm{x})$.
These can be related either by the familiar recursive formula
\begin{equation}
\label{newton-recurrence}
\sigma_m=\sum_{\mu=1}^k(-1)^{\mu-1}\sigma_{m-\mu}e_{\mu},
\end{equation}
or by the explicit evaluations~\cite[\S I1.2]{Macdonald1995}
\begin{equation}
\label{non-recursive}
\begin{aligned}
e_m&=(-1)^m\sum_{\Sigma_iik_i=m}\prod_{i=1}^m\frac{(-\sigma_i)^{k_i}}{k_i!i^{k_i}},\\
\sigma_m&=(-1)^mm\sum_{\Sigma_i ik_i=m}(\Sigma k_i-1)!\prod_{i=1}^m\frac{(-e_i)^{k_i}}{k_i!}
\end{aligned}
\end{equation}
where, consistent with the definition~\eqref{elementary-symmetric}, $e_m=0$ for $m>k$.

For every $\bm{x}\in(\mathbb{Z}/p^n\mathbb{Z})^k$, let
\[ \Delta(\bm{x})=\min\big(n,\min_{1\leqslant m\leqslant k}\ord_p\sigma_m(\bm{x})\big),\quad \Delta_0(\bm{x})=\min\Big(n,\min_{1\leqslant m\leqslant k}\ord_p\frac{\sigma_m(\bm{x})}{m!}\Big). \]

\begin{lemma}
\label{recurrence-lemma}
For every $k\in\mathbb{N}$, there exist constants $\rho_k(p)\in\mathbb{Z}_{\geqslant 0}$ such that $\rho_k(p)=0$ for all sufficiently large $p\geqslant p_0(k)$, and such that, for every $\bm{x}\in\mathbb{Z}^k$ and every $m>k$:
\begin{enumerate}
\item\label{reclemma-item1} if $\Delta_0(\bm{x})\geqslant 1$, then $\displaystyle\ord_p\frac{\sigma_m(\bm{x})}{m!}\geqslant\Delta_0(\bm{x})-(m-k)(1+\rho_k(p))$;
\item\label{reclemma-item2} if $\Delta(\bm{x})\geqslant\rho_k(p)$, then $\ord_p\sigma_m(\bm{x})\geqslant 2\Delta(\bm{x})-\rho_k(p)$.
\end{enumerate}
 \end{lemma}
 
\begin{proof}
We begin with item \eqref{reclemma-item1}. First, we claim that $p^{\Delta_0}\mid e_m$ for every $1\leqslant m\leqslant k$. Recall the elementary estimate $\ord_p(k!)<k/(p-1)$. Already every individual factor in the first expression in \eqref{non-recursive} satisfies
\[ \ord_p\frac{(-\sigma_i)^{k_i}}{k_i!i^{k_i}}>k_i\Big(\Delta_0-\frac1{p-1}\Big)=\Delta_0+\Big(\Delta_0(k_i-1)-\frac{k_i}{p-1}\Big), \]
from which the left-hand side is at least $\Delta_0$ and thus $p^{\Delta_0}\mid e_m$ for every $1\leqslant m\leqslant k$. Using this conclusion, the same argument shows that already every individual factor in the second expression in \eqref{non-recursive} is divisible by $p^{\Delta_0}$. From this it follows that
\[ \ord_p\frac{\sigma_m}{m!}\geqslant\Delta_0-\ord_pm!>\Delta_0-\frac{m}{p-1}\geqslant\Delta_0-(m-k)(1+\rho_k(p)) \]
with $\rho_k(p)=0$ for
\[ \text{either}\qquad m\geqslant 2k,\,\,p\geqslant 3\qquad\text{or}\qquad m<2k,\,\,p\geqslant 2k. \]
It remains to cover the cases when ($p=2$ or) $2<p<2k$ and $m<2k$, in which case we may simply take, say, $\rho_k(2)=k$ and $\rho_k(p)=\max\big(0,\max_{k<m<2k}\big\lceil(\ord_pm!)/(m-k)\big\rceil-1\big)$ for $3\leqslant p<2k$.

As for \eqref{reclemma-item2}, we may take $\rho_k(p)$ denote the highest power of $p$ appearing in any of the denominators on the right-hand side of the first expression in \eqref{non-recursive} over all $m\leqslant k$; thus, $\rho_k(p)\geqslant 0$, and $\rho_k(p)=0$ for all $p>k$. Then, of course,
\[ \ord_pe_m(\bm{x})\geqslant\Delta(\bm{x})-\rho_k(p). \]
The claim that $\ord_p\sigma_m(\bm{x})\geqslant 2\Delta(\bm{x})-\rho_k(p)\geqslant\Delta(\bm{x})$ now follows by induction on $m$ and applying the recurrence relation \eqref{newton-recurrence}. In fact, with just a little more care, induction shows the stronger inequality
\[ \ord\nolimits_p\sigma_m(\bm{x})\geqslant\Delta(\bm{x})+\left\lfloor\frac {m-1}k\right\rfloor(\Delta(\bm{x})-\rho_k(p)). \qedhere \]
\end{proof}

A useful variation of Lemma~\ref{recurrence-lemma} is the following result. For any two tuples $\bm{a},\bm{b}\in(p\mathbb{Z}/p^n\mathbb{Z})^k$ and any $j\geqslant 1$, denote
\[ \Delta_{j}(\bm{a},\bm{b})=\min(n,\ord_p(\sigma_j(\bm{b})-\sigma_j(\bm{a}))), \]
so that $\min_{1\leqslant j\leqslant k}\Delta_j(\bm{a},\bm{b})=\Delta(\bm{a},\bm{b})$ in the notation of Definition~\ref{alignment-partial}.

\begin{lemma}
\label{recurrence-lemma-plus-p}
There exist constants $\rho_k(p)\in\mathbb{Z}_{\geqslant 0}$ such that $\rho_k(p)=0$ for all sufficiently large $p\geqslant p_0(k)$ and such that:
\begin{alignat*}{3}
\Delta_m(\bm{a},\bm{b})&\geqslant\Delta(\bm{a},\bm{b})-\rho_k(p)&\quad&(\bm{a},\bm{b}\in(\mathbb{Z}/p^n\mathbb{Z})^k,\,\,m\in\mathbb{N}),\\
\Delta_m(\bm{a},\bm{b})&\geqslant\min(n,\Delta(\bm{a},\bm{b})+1-\rho_k(p))&\quad&(\bm{a},\bm{b}\in(p\mathbb{Z}/p^n\mathbb{Z})^k,\,\,m>k),\\
\Delta_m(\bm{a},\bm{b})&\geqslant\min(n,\Delta(\bm{a},\bm{b})+m-k-\lfloor\log_pm\rfloor-\rho_k(p))&&(\bm{a},\bm{b}\in(p\mathbb{Z}/p^n\mathbb{Z})^k,\,\,m>k).
\end{alignat*}
\end{lemma}

\begin{proof}
Write for brevity $\Delta_k=\Delta(\bm{a},\bm{b})$, so that $\sigma_m(\bm{a})\equiv\sigma_m(\bm{b})\pmod{p^{\Delta_k}}$ for $1\leqslant m\leqslant k$.
As in the proof of Lemma~\ref{recurrence-lemma} and with the same choice of $\rho_k(p)$, from the first expression in \eqref{non-recursive}, we first obtain that
\begin{align*}
e_m(\bm{a})-e_m(\bm{b})
&=(-1)^m\sum_{\sum_iik_i=m}\frac1{\prod_{i=1}^mk_i!i^{k_i}}\bigg(\prod_{i=1}^m(-\sigma_i(\bm{a}))^{k_i}-\prod_{i=1}^m(-\sigma_i(\bm{b}))^{k_i}\bigg)\\
&\equiv 0\pmod{p^{\Delta_k-\rho_k(p)}}\qquad(1\leqslant m\leqslant k),
\end{align*}
and then, using Newton's recurrences \eqref{newton-recurrence}, that also
\[ \sigma_m(\bm{a})\equiv\sigma_m(\bm{b})\pmod{p^{\Delta_k-\rho_k(p)}}\quad (m\in\mathbb{N}). \]
This establishes the first claim of the lemma.

Now, from Newton's recurrences \eqref{newton-recurrence} we have that
\begin{align*}
\sigma_m(\bm{a})-\sigma_m(\bm{b})
&=\sum_{\mu=1}^k(-1)^{\mu-1}\big(\sigma_{m-\mu}(\bm{a})e_{\mu}(\bm{a})-\sigma_{m-\mu}(\bm{b})e_{\mu}(\bm{b})\big)\\
&=\sum_{\mu=1}^k(-1)^{\mu-1}\big(\sigma_{m-\mu}(\bm{a})(e_{\mu}(\bm{a})-e_{\mu}(\bm{b}))+e_{\mu}(\bm{b})(\sigma_{m-\mu}(\bm{a})-\sigma_{m-\mu}(\bm{b})\big),
\end{align*}
and the second claim of the lemma follows by induction, keeping in mind that in this case $p\mid\sigma_{m-\mu}(\bm{a})$ and $p\mid e_{\mu}(\bm{b})$.

On the other hand, from the second expression in \eqref{non-recursive}, we also have that
\begin{gather*}
\sigma_m(\bm{a})-\sigma_m(\bm{b})
=(-1)^mm\sum_{\sum_iik_i=m}\frac{(\Sigma k_i-1)!}{\prod_{i=1}^mk_i!}\delta_{\bm{k}}(\bm{a},\bm{b}),\\
\delta_{\bm{k}}(\bm{a},\bm{b}):=\sum_{j=1}^{m}\Big[\big((-e_j(\bm{a}))^{k_j}-(-e_j(\bm{b}))^{k_j}\big)\prod_{j<u\leqslant m}(-e_u(\bm{a}))^{k_u}\prod_{1\leqslant v<j}(-e_v(\bm{b}))^{k_v}\Big].
\end{gather*}
We note that $p^i\mid e_i(\bm{a}),e_i(\bm{b})$ and so
\begin{align*}
\ord_p\delta_{\bm{k}}(\bm{a},\bm{b})
&\geqslant \Delta_k-\rho_k(p)+\min_{1\leqslant j\leqslant\min(k,m)}\Big(\sum_{u\neq j}uk_u+j(k_j-1)\Big)\\
&\geqslant\Delta_k+m-k-\rho_k'(p)
\end{align*}
for some constants $\rho_k(p)\in\mathbb{Z}_{\geqslant 0}$ such that $\rho_k'(p)=0$ for all sufficiently large $p\geqslant p_0(k)$. The third claim of the lemma follows from this by observing that
\[ \ord_p\frac{(\Sigma k_i-1)!}{\prod_{i=1}^mk_i!}\geqslant\ord_p\binom{\Sigma k_i}{k_1,\dots,k_m}\frac1{\Sigma k_i}\geqslant -\ord_p(\Sigma k_i)\geqslant -\lfloor\log_pm\rfloor. \qedhere \]
\end{proof}

\subsection{Class of phases}
\label{class-subsection}

Our method applies to exponential sums with phases which allow for sufficient application of the method of stationary phase of the previous subsection and are well-controlled by a a specified number of initial terms in a differentiability property of the form like~\eqref{sqroot-diff}. We formalize these requirements in the following definition.

\begin{definition}
\label{Fk-def}
Let $k\geqslant 1$, and let $\rho_k(p)\in\mathbb{Z}_{\geqslant 0}$ be any fixed choice of constants such that $\rho_k(p)=0$ for all sufficiently large primes $p\geqslant p_0(k)$.
Further, let $1+\rho_k(p)=\kappa_0=\kappa_0(F)\leqslant n$, and let $A_F\subseteq\mathbb{Z}/p^n\mathbb{Z}$ be a set invariant under translations by $p^{\kappa_0}\mathbb{Z}/p^n\mathbb{Z}$. We say that a function $F:A_F\to\mathbb{Z}/p^n\mathbb{Z}$ belongs to the class $\mathcal{F}_k(\mathbb{Z}/p^n\mathbb{Z})$ if it satisfies the following two conditions:
\begin{enumerate}
\item\label{Fk-def-item1} There exists a system of functions $\tilde{F}^{(m)}:A_F\to\mathbb{Z}/p^n\mathbb{Z}$ ($0\leqslant m<n$) with $\tilde{F}^{(0)}=F$ such that, denoting $F^{(m)}(x)=m!\tilde{F}^{(m)}(x)$ and $F^{(m+i)}(x)/i!=(m+i)_m\tilde{F}^{(m+i)}(x)$ ($m,i\geqslant 0$), we have that for every $m\in\mathbb{Z}_{\geqslant 0}$, $x\in A_F$, $\kappa\geqslant\kappa_0$, $t\in\mathbb{Z}/p^n\mathbb{Z}$, and $1\leqslant j<n$,
\begin{equation}
\label{diff-like}
F^{(m)}(x+p^{\kappa}t)\equiv\sum_{i=0}^j\frac{F^{(m+i)}(x)}{i!}p^{i\kappa}t^i\pmod{p^{\min((j+1)\kappa,n)}}.
\end{equation}
\item\label{Fk-def-item2} For every $x\in A_F$, denoting by
\[ \Delta_k(F;x)=\min_{1\leqslant i\leqslant k}\min\Big(n,\ord\nolimits_p\frac{F^{(i)}(x)}{i!}\Big)\in\mathbb{Z}, \]
we have that  $\displaystyle p^{\min(n,\Delta_k(F;x)+1)}\mid p^{(i-k)^{+}\kappa_0}\frac{F^{(i)}(x)}{i!}$ for every $i>k$.
\end{enumerate}
For an $F\in\mathcal{F}_k(\mathbb{Z}/p^n\mathbb{Z})$, we also write
\[ \Delta_k(F):=\max_{x\in A_F}\Delta_k(F;x). \]
\end{definition}
We note that the presence of the factor $p^{i\kappa}$ renders the fraction on the right hand side of \eqref{diff-like} unambiguous.

In light of \eqref{sqroot-diff}, for every $m\in\mathbb{Z}$, the function $s_m:(\mathbb{Z}/p^n\mathbb{Z})^{\times 2}\to\mathbb{Z}/p^n\mathbb{Z}$ given by $s_m(x)=x_{1/2}^m$ (for which $\kappa_0=1$ for $p>2$ and $\kappa_0=3$ for $p=2$) satisfies item \eqref{Fk-def-item1} with $s_m^{(i)}=(m/2)_is_{m-2i}$ and item \eqref{Fk-def-item2} with $\Delta_1(s_m;x)=\ord_pm$, and so $s_m\in\mathcal{F}_1(\mathbb{Z}/p^n\mathbb{Z})$.

As another example, for any two $k$-tuples $\bm{a}=(a_1,\dots,a_k)$ and $\bm{c}=(c_1,\dots,c_k)$ in $\mathbb{Z}/p^n\mathbb{Z}$ such that the domain
\[ X_{\bm{a}}=\big\{x\in\mathbb{Z}/p^n\mathbb{Z}:x+a_i\in(\mathbb{Z}/p^n\mathbb{Z})^{\times 2}\text{ for all }1\leqslant i\leqslant k\big\} \]
is non-empty, consider the function $f_{\bm{c},\bm{a},m}:X_{\bm{a}}\to\mathbb{Z}/p^n\mathbb{Z}$ defined by
\[ f_{\bm{c},\bm{a},m}(x)=\sum_{j=1}^kc_j(x+a_j)_{1/2}^{m}. \]
As above, this satisfies item \eqref{Fk-def-item1} with $f_{\bm{c},\bm{a},m}^{(i)}=(m/2)_if_{\bm{c},\bm{a},m-2i}$.
Considering the polynomial $P_{x,\bm{a}}(X)=\prod_{j=1}^k(X-(x+a_j)^{-1})=\sum_{\ell=0}^k(-1)^{\ell}e_{\ell}X^{k-{\ell}}\in(\mathbb{Z}/p^n\mathbb{Z})[X]$ with $e_0=1$, we have as in \eqref{newton-recurrence} the recurrence relations
\begin{align*}
f_{\bm{c},\bm{a},m-2k}&=\sum_{j=0}^{k-1}(-1)^{k-j-1}e_{k-j}f_{\bm{c},\bm{a},m-2j},\\
f_{\bm{c},\bm{a},m}^{(\ell+k)}&=\sum_{j=0}^{k-1}(-1)^{k-j-1}e_{k-j}(m/2-j-\ell)_{k-j}f_{\bm{c},\bm{a},m}^{(\ell+j)},
\end{align*}
from which it immediately follows that $p^{\Delta_k(f_{\bm{c},\bm{a},m};x)}\mid f_{\bm{c},\bm{a},m}^{(i)}(x)$ for every $i>k$ as well. For $i\geqslant 2k$ and $p\geqslant 3$, using the elementary estimate $\ord_p i!<i/(p-1)\leqslant i/2\leqslant i-k$, we find that
\[ p^{\Delta_k(f_{\bm{c},\bm{a},m};x)+1}\mid p^{(i-k)(1+\rho_k(p))}\frac{f_{\bm{c},\bm{a},m}^{(i)}(x)}{i!} \]
already with $\rho_k(p)=0$. It remains to consider the cases when $p=2$ or $k<i<2k$, in which case the same relation holds with $\rho_k(2)=k$ and
\[ \rho_k(p)=\max_{k<i<2k}\max\Big(0,\Big\lceil\frac{\ord\nolimits_pi!+1}{i-k}-1\Big\rceil\Big) \]
for $p\geqslant 3$.
We see that indeed $\rho_k(p)=0$ for $p\geqslant 2k$, so that $f_{\bm{c},\bm{a},m}\in\mathcal{F}_k(\mathbb{Z}/p^n\mathbb{Z})$. In the proof of Theorem~\ref{main-theorem}, we will encounter phases of the form $f_{\bm{c},\bm{a},1}$ with the additional condition that $\bm{c}\in\{\pm 1\}^k$, in which case we will see that more can be said.

\section{The alignment argument}
\label{alignment-section}

\subsection{Alignment modulo \texorpdfstring{$p$}{p}}
\label{alignment-p-sec}
For any two vectors $\bm{a}=(a_i)\in(\mathbb{Z}/p^n\mathbb{Z})^k$ and $\bm{\epsilon}=(\epsilon_i)\in\{\pm 1\}^k$, we consider the following system of $\ell$ congruences modulo $p^r$:
\begin{equation}
\label{alignment-section-system}
\Sigma_{\bm{\epsilon}}(k,\ell;\bm{a};p^r):\qquad \sum_{i=1}^k\epsilon_i(x+a_i)_{1/2}^{-1-2m}\equiv 0\pmod {p^r}\quad (0\leqslant m\leqslant\ell -1).
\end{equation}

\begin{lemma}
\label{alignment-mod-p-lemma}
For every $k\in\mathbb{N}$, there exists a constant $\rho_k(p)\in\mathbb{Z}_{\geqslant 0}$ such that $\rho_k(p)=0$ for all sufficiently large $p\geqslant p_0(k)$, and such that if the system $\Sigma_{\bm{\epsilon}}(k,\lceil k/2\rceil;\bm{a};p^{1+\rho_k(p)})$ has solutions, then $2\mid k$ and
\[ \bm{a}\equiv\sigma(\bm{a})\pmod p\quad\text{and}\quad \bm{\epsilon}=-\sigma(\bm{\epsilon}) \]
for some product of 2-cycles $\sigma\in S_k$.
\end{lemma}

\begin{proof}
Temporarily denoting $y_i=\epsilon_i(x+a_i)_{1/2}^{-1}$,
we may rewrite the given system of congruences $\Sigma_{\bm{\epsilon}}(k,\lceil k/2\rceil;\bm{a};p^{1+\rho_k(p)})$ as
\[ \sum_{i=1}^{k}y_i^{2m+1}\equiv 0\pmod{p^{1+\rho_k(p)}}\quad(0\leqslant m\leqslant \lceil k/2\rceil-1). \]
This further means that
\begin{equation}
\label{power-means-agreement}
\sum_{i=1}^{k}y_i^m\equiv\sum_{i=1}^{k}(-y_i)^m\pmod{p^{1+\rho_k(p)}}\quad(1\leqslant m\leqslant 2\lceil k/2\rceil),
\end{equation}
thus in particular for $1\leqslant m\leqslant k$. According to Newton's formulae~\eqref{newton-recurrence}, the elementary symmetric polynomials
\[ e_m(\bm{y})=\sum_{1\leqslant j_1<j_2<\dots<j_m\leqslant k}y_{j_1}y_{j_2}\cdots y_{j_m} \]
can be expressed as polynomials in the power means $s_{m'}(\bm{y})=\sum_{i=1}^{k}y_i^{m'}$ ($m'\leqslant m$), with rational coefficients which (in particular whose denominators) depend on $k$ only. Therefore, for $\rho_k(p)\geqslant 0$ sufficiently large in terms of $k$ and such that $\rho_k(p)=0$ for all sufficiently large $p$ (namely, for those $p$ not appearing at all in the denominators of any of the $k$ polynomials under consideration), the congruences \eqref{power-means-agreement} imply that
\[ e_m(\bm{y})\equiv e_m(-\bm{y})\pmod p\quad (1\leqslant m\leqslant k). \]
This conclusion implies that
\[ \prod_{i=1}^{k}(X-y_i)=\prod_{i=1}^{k}(X+y_i)\quad\text{as polynomials in }(\mathbb{Z}/p\mathbb{Z})[X], \]
and therefore by unique factorization there exists a permutation $\pi\in S_{k}$ such that
\begin{equation}
\label{alignment-yi}
y_i\equiv -y_{\pi(i)}\pmod p\quad (1\leqslant i\leqslant k).
\end{equation}

Now, going back to the definition of $y_i=\epsilon_i(x+a_i)_{1/2}^{-1}$, we recall that the definition of $s(x)=x_{1/2}$ in \S\ref{sqroot-section} precludes the possibility that $x_{1/2}\equiv -x'_{1/2}\pmod p$ for any $x,x'\in(\mathbb{Z}/p^n\mathbb{Z})^{\times 2}$ (since this would imply $x\equiv x'\pmod p$ and therefore $x_{1/2}\equiv s(x)\equiv s(x')\equiv x'_{1/2}\pmod p$). Putting everything together, the alignment \eqref{alignment-yi} is only possible if all cycles of $\pi$ are of even length (in particular, $2\mid k$) and, for every $1\leqslant i\leqslant k$,
\[ \epsilon_i=-\epsilon_{\pi(i)},\quad \epsilon_iy_i=(x+a_i)_{1/2}^{-1}\equiv -(-(x+a_{\pi(i)})_{1/2}^{-1})\pmod p. \]
This means that also $a_i\equiv a_{\pi(i)}\pmod p$. Writing $\pi$ as the product of disjoint cycles $(a_{i_1}a_{i_2}\cdots a_{i_{2\ell-1}}a_{i_{2\ell}})$, the statement of the lemma follows by taking $\sigma$ to be the corresponding product of products of 2-cycles $(a_{i_1}a_{i_2})\cdots(a_{i_{2\ell-1}}a_{i_{2\ell}})$.
\end{proof}

We note that a substantially more straightforward argument using the Vandermonde determinant suffices to show that the same conclusion follows if the system $\Sigma_{\bm{\epsilon}}(k,k;\bm{a};p^{1+\rho_k(p)})$ is consistent (with a suitable choice of $\rho_k(p)\in\mathbb{Z}_{\geqslant 0}$ such that $\rho_k(p)=0$ for all $p\geqslant p_0(k)$).

\subsection{Comparison of measures of alignment}
In this section, we show that the alignment \eqref{symm-poly-alignment} (essentially an alternative expression of $\Delta=\Delta(\bm{a},\bm{b})$ forces a pairwise alignment of two tuples $\bm{a}$ and $\bm{b}$ at least modulo $p^{\Delta^0}$ with $\Delta^0\geqslant\Delta/k$.

\begin{lemma}
\label{1-over-k}
Suppose that for two tuples $\bm{a},\bm{b}\in(\mathbb{Z}/p^n\mathbb{Z})^k$ we have the congruence of polynomials
\begin{equation}
\label{symm-poly-alignment}
(X-a_1)(X-a_2)\cdots (X-a_k)\equiv (X-b_1)(X-b_2)\cdots(X-b_k)\pmod{p^{\Delta}}
\end{equation}
for some $1\leqslant\Delta\leqslant n$. Then there exists a permutation $\sigma\in S_k$ such that
\[ \bm{b}\equiv\sigma(\bm{a})\pmod{p^{\lceil\Delta/k\rceil}}. \] 
\end{lemma}

\begin{proof}
We prove by induction on $1\leqslant m\leqslant\lceil\Delta/k\rceil$ that there exists a permutation $\sigma_m\in S_k$ such that $\bm{b}\equiv\sigma_m(\bm{a})\pmod{p^m}$. For $m=1$, the claim follows from unique factorization of polynomials over the field $\mathbb{Z}/p\mathbb{Z}$.

Suppose the claim has been proved for some $m<\Delta/k$. By relabeling, we may assume that $\bm{b}\equiv\bm{a}\pmod{p^m}$, and thus there exist a partition $[k]=\bigsqcup_{i=1}^rJ_i$ into $1\leqslant r\leqslant p^m$ disjoint non-empty sets and representatives $x_i$ ($1\leqslant i\leqslant r$) such that
\[ a_j,b_j\equiv x_i\pmod{p^m}\quad(j\in J_i) \]
and such that $x_i\neq x_{i'}\pmod{p^m}$ for every $i\neq i'$. We fix an $1\leqslant i\leqslant r$, denote $J=J_i$, and substitute $X=x_i+p^mY$ into \eqref{symm-poly-alignment}: this yields
\begin{align*}
&\prod_{j\in J}(x_i-a_j+p^mY)\prod_{j\not\in J}(x_i-a_j+p^mY)\\
&\qquad\qquad\equiv \prod_{j\in J}(x_i-b_j+p^mY)\prod_{j\not\in J}(x_i-b_j+p^mY)\pmod{p^{\Delta}}.
\end{align*}
We substitute $a_j=x_i+p^m\alpha_{j_i}$ and $b_j=x_i+p^m\beta_{j_i}$ for $j\in J$, factor out $p^m$ from each of these $|J|$ factors and power $p^{\ord_p(x_i-a_j)}$ from each of the remaining $k-|J|$ factors, and extract the highest power of $p$ dividing each side of the above congruence. This highest power on the left- and right-hand side satisfies
\[ \Delta_m^{-,l}\Delta_m^{-,r}\leqslant m|J|+(m-1)(k-|J|)<\Delta-k+|J|\leqslant\Delta, \]
so we may divide the congruence by $p^{\Delta_m^{-}}$ with $\Delta_m^{-}=\min(\Delta_m^{-,l},\Delta_m^{-,r})<\Delta$, and then reduce the remaining congruence modulo $p^{\Delta-\Delta_m^{-}}$ to modulus $p$. The remaining congruence modulo $p$ is of degree $|J|$; by  comparing the leading coefficients, we find that in fact $\Delta_m^{-,l}=\Delta_m^{-,r}$, the leading coefficients agree modulo $p$, and
\[ \prod_{j\in J}(Y-\alpha_j)\equiv\prod_{j\in J}(Y-\beta_j)\pmod p. \]
By the unique factorization of polynomials over the field $\mathbb{Z}/p\mathbb{Z}$, this implies that there exists a permutation $\sigma_{m+1,J}\in S_J$ such that
\[ \bm{\beta}_{|J}\equiv\sigma_{m+1,J}(\bm{\alpha}_{|J})\pmod p,\quad\text{that is,}\quad
\bm{b}_{|J}\equiv\sigma_{m+1,J}(\bm{a}_{|J})\pmod{p^{m+1}}, \]
where we denote shorthand $\bm{\alpha}_{|J}=(\alpha_j)_{j\in J}$ and similarly for $\bm{\beta}_{|J}$, $\bm{a}_{|J}$, and $\bm{b}_{|J}$.

Combining the resulting permutations $\sigma_{m+1,J_i}$ over all $1\leqslant i\leqslant r$ produces a permutation $\sigma_{m+1}\in S_k$ such that $\bm{b}\equiv\sigma_{m+1}(\bm{a})\pmod{p^{m+1}}$, completing the inductive proof.
\end{proof}

\subsection{A determinant evaluation}
\label{determinant-subsection}

We will need the following explicit determinant evaluation. For $1\leqslant\ell\leqslant k$, consider the $k\times\ell$ rectangular matrix
\begin{align*}
B_k(a,\ell, X)
&=
\left[\binom{a-i+1}j X^{i+j-1}\right]_{\substack{1\leqslant i\leqslant k,\\1\leqslant j\leqslant\ell}}\\
&=\left[
\begin{array}{r@{\,}l @{\quad} r@{\,}l @{\quad} c @{\quad} r@{\,}l}
  \binom{a}{1} & X & \binom{a}{2} & X^2 & \cdots & \binom{a}{\ell} & X^{\ell} \\
  \binom{a-1}{1} & X^2 & \binom{a-1}{2} & X^3 & \cdots & \binom{a-1}{\ell} & X^{\ell+1} \\
  \binom{a-2}{1} & X^3 & \binom{a-2}{2} & X^4 & \cdots & \binom{a-2}{\ell} & X^{\ell+2} \\
  \vdots & & \vdots & & \ddots & \vdots & \\
  \binom{a-k+1}{1} & X^k & \binom{a-k+1}{2} & X^{k+1} & \cdots & \binom{a-k+1}{\ell} & X^{\ell+k-1} \\
\end{array}
\right].
\end{align*}
Given a partition $\bm{k}=(k_1,k_2,\dots,k_m)\in\mathbb{N}^m$ of $k=k_1+k_2+\dots+k_m$, a parameter $a$ (in whatever ring we are working over) and a vector of variables $\bm{X}=(X_1,X_2,\dots,X_m)$, consider the block $k\times k$ matrix
\begin{equation}
\label{block-matrix}
B(a,\bm{k};\bm{X})=\begin{bmatrix} B_k(a,k_1,X_1) & B_k(a,k_2,X_2) &\cdots & B_k(a,k_m,X_m) \end{bmatrix}.
\end{equation}
Then, we have the following evaluation.

\begin{lemma}
\label{determinant-lemma}
The block matrix $B(a,\bm{k};\bm{X})$ defined in \eqref{block-matrix} satisfies
\begin{equation}
\label{dlemma-toprove}
\det B(a,\bm{k};\bm{X})=(-1)^{\sum\limits_{i=1}^m\binom{k_i}2}\binom{a}{k}\binom{k}{k_1,k_2,\dots,k_m}\prod_{i=1}^mX_i^{k_i^2}\prod_{1\leqslant i<j\leqslant m}(X_j-X_i)^{k_ik_j}.
\end{equation}
\end{lemma}

\begin{proof}
We will use the condensation method, whose basic identity (often attributed to C.L.Dodgson) we cite from \cite[Proposition 10]{Krattenthaler1999a}. For an $n\times n$ matrix $A$, denoting by $A_{i_1,\dots,i_k}^{j_1,\dots,j_k}$ the submatrix in which rows $i_1,\dots,i_k$ and $j_1,\dots,j_k$ have been removed, we have the identity
\begin{equation}
\label{condensation}
\det A\cdot\det A_{1,n}^{1,n}=\det A_1^1\cdot\det A_n^n-\det A_1^n\cdot\det A_n^1.
\end{equation}

We first address the case $m=1$, when there is only one block, and we evaluate more generally that, for $b\in\mathbb{Z}_{\geqslant 0}$, the determinant of the matrix
\[ M_k(a,b)=\left[\binom{a-i+1}{b+j}\right]_{1\leqslant i,j\leqslant k}\]
satisfies
\begin{equation}
\label{claimed-det-m1}
\det M_k(a,b)=(-1)^{\binom k2}\prod_{c=0}^{b}\binom{a-c}k\binom{c+k}k^{-1}=(-1)^{\binom{k}2+k(b+1)}\prod_{c=0}^b\frac{(a-c)_k}{(-c-1)_k}.
\end{equation}
Denoting the right-hand side of the claimed equality \eqref{claimed-det-m1} by $m_k(a,b)$, we see that
\begin{align*}
\frac{m_{k+1}(a,b)}{m_k(a,b)}&=(-1)^{k+b+1}\prod_{c=0}^b\frac{a-c-k}{-c-1-k}=(-1)^{k+b+1}\frac{(a-k)_{b+1}}{(-k-1)_{b+1}},\\
\frac{m_{k+1}(a+1,b)}{m_k(a,b)}&=(-1)^{k+b+1}\prod_{c=0}^b\frac{a+1-c}{-c-1-k}=(-1)^{k+b+1}\frac{(a+1)_{b+1}}{(-k-1)_{b+1}},
\end{align*}
from which it follows that
\begin{align*}
&\frac{m_{k-1}(a-1,b+1)m_{k-1}(a,b)-m_{k-1}(a-1,b)m_{k-1}(a,b+1)}{m_k(a,b)m_{k-2}(a-1,b+1)}\\
&\qquad =\frac{(a-k+1)_{b+2}}{(-k+1)_{b+2}}\frac{(-k)_{b+1}}{(a-k+1)_{b+1}}-\frac{(a)_{b+2}}{(-k+1)_{b+2}}\frac{(-k)_{b+1}}{(a)_{b+1}}
=\frac{-k+1}{-k+1}=1.
\end{align*}
Now, the claim \eqref{claimed-det-m1} is true for $k=1$ and $k=2$ by direct verification, and then it follows by induction as \eqref{condensation} shows that indeed $\det M_k(a,b)$ equals
\begin{align*}
&\frac{\det M_{k-1}(a-1,b+1)\det M_{k-1}(a,b)-\det M_{k-1}(a-1,b)\det M_{k-1}(a,b+1)}{\det M_{k-2}(a-1,b+1)}\\
&\qquad=\frac{m_{k-1}(a-1,b+1)m_{k-1}(a,b)-m_{k-1}(a-1,b)m_{k-1}(a,b+1)}{m_{k-2}(a-1,b+1)}=m_k(a,b).
\end{align*}
From \eqref{claimed-det-m1}, the claim \eqref{dlemma-toprove} for $m=1$ follows from $\det B(a,(k);(X))=X^{k^2}\det M_k(a,0)=(-1)^{\binom k2}\binom ak X^{k^2}$.

To establish the general case, we proceed by induction on $k$, the base case $k=1$ being trivial. For the inductive step, we may assume that $m\geqslant 2$, since the theorem is already proved for $m=1$. In this case, denote 
\[ \bm{k}''=(\dots,k_{m-1}-1,k_m-1),\quad \bm{k}'{}^{\circ}=(\dots,k_{m-1}-1,k_m),\quad \bm{k}^{\circ}{}'=(\dots,k_{m-1},k_m-1). \]
 We apply the condensation equality \eqref{condensation} with rows $1$ and $k$ and the final columns of the blocks $(m-1)$ and $m$ (say, columns $k'$ and $k$). Doing so and denoting $B=B(a,\bm{k},\bm{X})$, we encounter determinants
\begin{equation}
\label{recursive-dets}
\begin{alignedat}{3}
\det B_1^{k'}&=XX_m\det B(a-1,\bm{k}'{}^{\circ},\bm{X}), &\qquad \det B_k^k&=\det B(a,\bm{k}^{\circ}{}',\bm{X}),\\
\det B_1^k&=XX_{m-1}\det B(a-1,\bm{k}^{\circ}{}',\bm{X}),& \det B_k^{k'}&=\det B(a,\bm{k}'{}^{\circ},\bm{X}),\\
\det B_{1,k}^{k',k}&=X\det B(a-1,\bm{k}'',\bm{X}),
\end{alignedat}
\end{equation}
where we also denoted shorthand $X=\prod_{i=1}^{m-2}X_i^{k_i}\cdot X_{m-1}^{k_{m-1}-1} X_m^{k_m-1}$.
Denoting by $b(a,\bm{k},\bm{X})$ the right-hand side of the claimed equality \eqref{dlemma-toprove}, we find
\begin{align*}
\frac{b(a-1,\bm{k}'{}^{\circ},\bm{X})}{b(a-1,\bm{k}'',\bm{X})}&=(-1)^{k_m-1}\frac{a-k+1}{k_m}X_m^{2k_m-1}\prod_{i=1}^{m-2}(X_m-X_i)^{k_i}(X_m-X_{m-1})^{k_{m-1}-1},\\
\frac{b(a,\bm{k},\bm{X})}{b(a,\bm{k}'{}^{\circ},\bm{X})}&=(-1)^{k_{m-1}-1}\frac{a-k+1}{k_{m-1}}X_{m-1}^{2k_{m-1}-1}\prod_{i=1}^{m-2}(X_{m-1}-X_i)^{k_i}(X_m-X_{m-1})^{k_{m}},
\end{align*}
and analogously for $b(a-1,\bm{k}^{\circ}{}',\bm{X})/b(a-1,\bm{k}'',\bm{X})$ and $b(a,\bm{k},\bm{X})/b(a,\bm{k}^{\circ}{}',\bm{X})$. Hence,
\[ \frac{b(a-1,\bm{k}'{}^{\circ},\bm{X})b(a,\bm{k}^{\circ}{}',\bm{X})}{b(a,\bm{k},\bm{X})b(a-1,\bm{k}'',\bm{X})}
=\frac{b(a-1,\bm{k}^{\circ}{}',\bm{X})b(a,\bm{k}'{}^{\circ},\bm{X})}{b(a,\bm{k},\bm{X})b(a-1,\bm{k}'',\bm{X})}
=\frac1{X_m-X_{m-1}} \]
Upon applying \eqref{condensation} and \eqref{recursive-dets} and invoking the inductive hypothesis, we indeed find that
\begin{align*}
\det B&=\frac{\det B_1^{k'}\det B_k^k-\det B_1^k\det B_k^{k'}}{\det B_{1,k}^{k',k}}\\
&=\frac{X_mb(a-1,\bm{k}'{}^{\circ},\bm{X})b(a,\bm{k}^{\circ}{}',\bm{X})-X_{m-1}b(a-1,\bm{k}^{\circ}{}',\bm{X})b(a,\bm{k}'{}^{\circ},\bm{X})}
{b(a-1,\bm{k}'',\bm{X})}\\
&=\frac{X_m-X_{m-1}}{X_m-X_{m-1}}b(a,\bm{k},\bm{X})=b(a,\bm{k},\bm{X}).
\end{align*}
This concludes the proof.
\end{proof}

It may well be that there is a more elementary evaluation of the above determinant or (even more likely) that it is well-known; however, we could not readily locate a reference and were grateful to learn about this method from the wonderful survey of Krattenthaler~\cite{Krattenthaler1999a}.

\subsection{Alignment in groups}

With Lemma~\ref{determinant-lemma} ready for use, we return to the question of understanding the solutions of the system of congruences
\[\Sigma_{\bm{\epsilon}}(k,\ell;\bm{a};p^r):\qquad \sum_{i=1}^k\epsilon_i(x+a_i)_{1/2}^{-1-2m}\equiv 0\pmod{p^r}\quad (0\leqslant m\leqslant\ell-1). \]
Assuming from now on that
\[ \ell\geqslant\lceil k/2\rceil,\quad 1+\rho_k(p)\leqslant r\leqslant n, \]
Lemma~\ref{alignment-mod-p-lemma} shows that the system $\Sigma_{\bm{\epsilon}}(k,\ell;\bm{a};p^r)$ has no solutions unless we can split the multiset $[a_1,a_2,\dots,a_k]$ into a disjoint union of non-empty sub-multisets according to their residues modulo $p$ and re-index the variables accordingly as $\bm{a}=(a_{ti})$, $\bm{\epsilon}=(\epsilon_{ti})$ with
\begin{equation}
\label{kt-def}
[a_1,a_2,\dots,a_k]=\bigsqcup_{t\in T}[a_{t1},\dots,a_{tk_t}],\quad a_{ti}\equiv t\pmod p\quad (1\leqslant i\leqslant k_t)
\end{equation}
for the set of attained residues $T\subseteq\mathbb{Z}/p\mathbb{Z}$, where $|T|\leqslant k$ and
\begin{equation}
\label{kt-def-splitting}
2\mid k_t,\quad \epsilon_{ti}=\begin{cases} 1,&i\in I_t=\{1,2,\dots,k_t/2\},\\ -1,&i\in I_t'=\{k_t/2+1,k_t/2+2,\dots,k_t\},\end{cases}\quad \begin{gathered}\relax [k_t]=I_t\sqcup I_t',\\|I_t|=|I_t'|=k_t/2.\end{gathered}
\end{equation}
We rewrite our system of congruences $\Sigma_{\bm{\epsilon}}(k,\ell;\bm{a};p^r)$ as
\[ \sum_{t\in T}\sum_{i=1}^{k_t}\epsilon_{ti}(x+a_{ti})_{1/2}^{-1-2m}\equiv 0\bmod{p^r}\quad (0\leqslant m\leqslant\ell-1). \]
Using the expansion \eqref{sqroot-diff}, the above may be rewritten as
\[ \sum_{t\in T}\sum_{i=1}^{k_t}\epsilon_{ti}\sum_{j=0}^n\binom{-1/2-m}{j}(x+t)_{1/2}^{-1-2(m+j)}(a_{ti}-t)^j\equiv 0\bmod{p^r}\quad (0\leqslant m\leqslant\ell-1). \]
Now, for every $t\in T$ and $j\geqslant 0$, denote
\begin{equation}
\label{deltattje-def}
\delta_{tj,\bm{\epsilon}}=\sum_{i=1}^{k_t}\epsilon_{ti}(a_{ti}-t)^j=\sum_{i\in I_t}(a_{ti}-t)^j-\sum_{i\in I_t'}(a_{ti}-t)^j,\quad \Delta_{tj,\bm{\epsilon}}=\ord_p\delta_{tj,\bm{\epsilon}}.
\end{equation}
We are about to invoke Lemma~\ref{recurrence-lemma-plus-p}. Let $\rho_k(p)$ be the values provided by that lemma, and pick once and for all an even integer $k^{\sharp}\in\mathbb{N}$ (sufficiently large and depending on $k$ only) such that
\[ j-(k+3)\lfloor\log_3(2j)\rfloor\geqslant k+\max_{p:\rho_k(p)\neq 0}\rho_k(p)\quad\text{for every }j\geqslant\ k^{\sharp}/2. \]
We now pick $\tilde{k}_t=k_t$ if $\rho_k(p)=0$ and $p>kk^{\sharp}$, and $\tilde{k}_t=k^{\sharp}$ otherwise. We also denote $\tilde{k}_T=\sum_{t\in T}\tilde{k}_t$ and $\tilde{k}=\max_T\tilde{k}_T$; thus, $\tilde{k}=\tilde{k}_T=k$ for all sufficiently large $p$ (namely those with $\rho_k(p)=0$ and $p>kk^{\sharp}\geqslant |T|k^{\sharp}$), and $\tilde{k}_T\leqslant |T|k^{\sharp}$ for all $p$.

The property these choices ensure is as follows. Let $\rho_{k}'(p)=\max_T\ord_p\binom{-1/2}{\tilde{k}_T/2}\binom{\tilde{k}_T/2}{(\tilde{k}_t/2)_{t\in T}}$. Then
\begin{align*}
\rho_{k}'(p)&\leqslant\ord_p(-\tfrac12)_{\tilde{k}_T/2}-\sum_{t\in T}\ord_p\big(\tilde{k}_t/2\big)!\\
&<\frac{\tilde{k}_T}{2(p-1)}-\sum_{t\in T}\frac{\tilde{k}_t}{2(p-1)}+\lfloor\log_p(\tilde{k}_T)\rfloor+\sum_{t\in T}\lfloor\log_p(\tilde{k}_t/2)\rfloor\\
&\leqslant (|T|+2)\lfloor\log_pk^{\sharp}\rfloor
\leqslant j-k-\lfloor\log_pj\rfloor-\rho_k(p)
\end{align*}
for every $j\geqslant k^{\sharp}/2$, as well as $\rho_{k}'(p)=0$ for $p>kk^{\sharp}$.
According to Lemma~\ref{recurrence-lemma-plus-p}, if $\rho_k(p)=0$ and $p>kk^{\sharp}$, then
\begin{equation}
\label{crazy}
\Delta_{t,\bm{\epsilon}}:=\min_{1\leqslant j\leqslant \tilde{k}_t/2}\Delta_{tj,\bm{\epsilon}}=\min_{j\geqslant 1}\Delta_{tj,\bm{\epsilon}},\quad \min_{j>\tilde{k}_t/2}\Delta_{tj,\bm{\epsilon}}>\Delta_{t,\bm{\epsilon}}+\rho_{k}'(p),
\end{equation}
simply because then $\Delta_{tj,\bm{\epsilon}}\geqslant 1+\Delta_{t,\bm{\epsilon}}$ for $j>\tilde{k}_t/2=k_t/2$ and $\rho_{k}'(p)=0$. But \eqref{crazy} holds otherwise as well, because then $\tilde{k}_t=k^{\sharp}$ and so for $j>\tilde{k}_t/2=k^{\sharp}/2$ we have again by Lemma~\ref{recurrence-lemma-plus-p} that
\[ \Delta_{tj,\bm{\epsilon}}\geqslant\Delta_{t,\bm{\epsilon}}+j-k-\lfloor\log_pj\rfloor-\rho_k(p)>\Delta_{t,\bm{\epsilon}}+\rho_{k}'(p). \]

At this point, we note that, while the values of $\delta_{tj,\bm{\epsilon}}$ depend on the choices of representatives $t\in T$ (which we have fixed once and for all), the values of $\Delta_{t,\bm{\epsilon}}$ do not and may even be computed by formally setting $t=0$ in the formulas \eqref{deltattje-def}, so that in fact
\[ \Delta(\bm{a}_t^{I_t},\bm{a}_t^{I_t'})-\rho_k(p)\leqslant\Delta_{t,\bm{\epsilon}}\leqslant \Delta(\bm{a}_t^{I_t},\bm{a}_t^{I_t'}) \]
in the notations of Definitions~\ref{alignment-partial} and \ref{Deltaast-maindef}.

Denoting
\begin{equation}
\label{deltaast-def}
\Delta^{\ast}_{\bm{\epsilon}}:=\min_{t\in T}\Delta_{t,\bm{\epsilon}},\quad \delta_{tj,\bm{\epsilon}}^{\ast}=p^{-\Delta^{\ast}_{\bm{\epsilon}}}\delta_{tj,\bm{\epsilon}},
\end{equation}
our system of congruences implies that
\begin{align*}
\smash[b]{p^{\Delta^{\ast}_{\bm{\epsilon}}}\sum_{t\in T}\sum_{j=1}^{\tilde{k}_t/2}\binom{-1/2-m}{j}(x+t)_{1/2}^{-1-2(m+j)}\delta_{tj,\bm{\epsilon}}^{\ast}}+p^{\Delta^{\ast}_{\bm{\epsilon}}+\rho_{k}'(p)+1}\Xi\equiv{}&0\pmod{p^r}\\
&(0\leqslant m\leqslant\ell-1)
\end{align*}
for some $\Xi,\delta_{tj,\bm{\epsilon}}^{\ast}\in\mathbb{Z}$ such that
\begin{equation}
\label{one-coprime}
\delta_{tj,\bm{\epsilon}}^{\ast}\not\equiv 0\pmod p\quad\text{for at least one }t\in T\text{ and }1\leqslant j\leqslant \tilde{k}_t/2.
\end{equation}

Such a system of congruences is trivially satisfied if $r\leqslant\Delta^{\ast}_{\bm{\epsilon}}$. The following lemma, which is the main result of this subsection, shows that, for $\ell\geqslant k/2+\tilde{\rho}_k(p)$ (with some $\tilde{\rho}_k(p)\in\mathbb{Z}_{\geqslant 0}$ such that $\tilde{\rho}_k(p)=0$ for sufficiently large $p$ as usual), this is in fact the only such case.

\begin{lemma}
\label{alignment-main}
Given $\bm{a}\in(\mathbb{Z}/p^n\mathbb{Z})^k$ and $\bm{\epsilon}\in\{\pm 1\}^k$, define $(k_t)_{t\in T}$ as in \eqref{kt-def}. For every $k\in\mathbb{N}$, there exists a constant $\rho_k(p)\in\mathbb{Z}_{\geqslant 0}$ such that $\rho_k(p)=0$ for all sufficiently large $p\geqslant p_0(k)$, and such that then the system of congruences $\Sigma_{\bm{\epsilon}}(k,\ell;\bm{a};p^r)$ has no solutions if either:
\begin{itemize}
\item $\ell\geqslant\lceil k/2\rceil$, $r>\rho_k(p)$, and $2\nmid k$ or the splitting conditions \eqref{kt-def-splitting} are not satisfied,~or
\item $\ell\geqslant\lceil k/2\rceil+\rho_k(p)$ and $r>\Delta^{\ast}_{\bm{\epsilon}}+\rho_k(p)$, where $\Delta^{\ast}_{\bm{\epsilon}}\in\mathbb{N}$ is defined as in \eqref{deltaast-def}, otherwise.
\end{itemize}
\end{lemma}

\begin{proof}
We have already established the claim in the case when $2\nmid k$ or the splitting conditions \eqref{kt-def-splitting} are not satisfied, so assume henceforth that $2\mid k$ and \eqref{kt-def-splitting} holds, define $\Delta^{\ast}_{\bm{\epsilon}}\in\mathbb{N}$ by \eqref{deltaast-def}, and assume that $r>\Delta^{\ast}_{\bm{\epsilon}}+\rho_{k}'(p)$. In light of the above discussion and \eqref{one-coprime},
every solution of the system $\Sigma_{\bm{\epsilon}}(k,\tilde{k}_T/2;\bm{a};p^r)$ leads to a solution of the homogeneous $(\tilde{k}_T/2)\times (\tilde{k}_T/2)$ system of congruences
\begin{equation}
\label{remaining-system}
\sum_{t\in T}\sum_{j=1}^{\tilde{k}_t/2}\binom{-1/2-m}{j}(x+t)_{1/2}^{-1-2(m+j)}\delta_{tj,\bm{\epsilon}}^{\ast}\equiv 0\!\!\pmod {p^{1+\rho_{k}'(p)}}\quad(0\leqslant m\leqslant\tilde{k}_T/2-1)
\end{equation}
in which not all $\delta_{tj,\bm{\epsilon}}^{\ast}$ are divisible by $p$.

We can recognize the matrix of this system as the block matrix of the form \eqref{block-matrix} studied in \S\ref{determinant-subsection}. It consists of $m=|T|$ blocks indexed by $t\in T$, each of size $(\tilde{k}_T/2)\times(\tilde{k}_t/2)$, and may be expressed in the notations of \S\ref{determinant-subsection} as
\[ B=\bigg[ \Big((x+t)_{1/2}^{-1}B_{\tilde{k}}(-\tfrac12,\tilde{k}_t/2,(x+t)^{-1})\Big)_{t\in T}\bigg]. \]
Using Lemma~\ref{determinant-lemma}, the determinant of the coefficient matrix of the system \eqref{remaining-system} satisfies
\[ |\det B|=\binom{-\tfrac12}{\tilde{k}_T/2}\binom{\tilde{k}_T/2}{(\tilde{k}_t/2)_{t\in T}}\prod_{t\in T}(x+t)^{-\tilde{k}_t^2/4-\tilde{k}_t/2}\prod_{\substack{t,t'\in T\\t\neq t'}}((x+t)^{-1}-(x+t')^{-1})^{\tilde{k}_t\tilde{k}_{t'}/4}, \]
and, in particular, $\ord_p\det B\leqslant\rho_{k}'(p)$. By Cram\'er's formulas, the system of congruences \eqref{remaining-system} modulo $p^{1+\rho_k'(p)}$ satisfies
\[ \delta_{tj,\bm{\epsilon}}^{\ast}\equiv 0\pmod{p}. \]
This contradicts our assumption that at least one $\delta_{tj,\bm{\epsilon}}^{\ast}$ is not divisible by $p$, and the claim follows by adjusting the values $\rho_k(p)$ to $\max(\rho_k(p),\rho_k'(p),(\tilde{k}-k)/2)$ and keeping in mind that $\tilde{k}=k$ for all sufficiently large $p$.
\end{proof}

\section{The stratification argument}
\label{stratification-section}

The main result of this section is Lemma~\ref{split-domain-proposition}, in which we execute the most sensitive part of our argument, which will serve to bound the number and contributions of the nearly stationary points in our exponential sums.

In preparation for this, let $k\in\mathbb{N}$, and let $F$ belong to the class $\mathcal{F}_k(\mathbb{Z}/p^n\mathbb{Z})$ introduced in Definition~\ref{Fk-def}, and recall the parameter $\kappa_0(F)=1+\rho_k(p)$, with constants $\rho_k(p)\in\mathbb{Z}_{\geqslant 0}$ such that $\rho_k(p)=0$ for all $p\geqslant p_0(k)$. For every $\kappa\geqslant\kappa_0(F)$, let
\begin{equation}
\label{Likappa-def}
\begin{alignedat}{5}
f_{\kappa}(x_0,i)&=\min\Big(n,\ord_p\frac{F^{(i)}(x_0)}{i!}+\kappa i\Big),&\quad& &L_{\kappa}(x_0)&=\min_{1\leqslant i\leqslant k}f_{\kappa}(x_0,i),\\
f_{\kappa}^{\flat}(x_0,i)&=\min\Big(n,\ord_pi+\ord_p\frac{F^{(i)}(x_0)}{i!}+\kappa i\Big),&&&
L_{\kappa}^{\flat}(x_0)&=\min_{1\leqslant i\leqslant k}f_{\kappa}^{\flat}(x_0,i),\\
I_{\kappa}^{\flat}(x_0)&=\argmin_{1\leqslant i\leqslant k}f_{\kappa}^{\flat}(x_0,i), &&&i_{\kappa}^{\flat}(x_0)&=\max I_{\kappa}^{\flat}(x_0).
\end{alignedat}
\end{equation}
Now, from the condition that $F\in\mathcal{F}_k(\mathbb{Z}/p^n\mathbb{Z})$, we know that in fact
\[ \argmin_{i\geqslant 1}f_{\kappa}(x_0,i)\subseteq[k],\quad L_{\kappa}(x_0)=\min_{i\geqslant 1}f_{\kappa}(x_0,i). \]
It will also be convenient to denote $r_k(p)=\max_{i\in[k]}\ord_pi$, noting that $r_k(p)=0$ for all $p>k$. With this notation, it follows from part \eqref{Fk-def-item2} of Definition~\ref{Fk-def} that, for $\kappa\geqslant\kappa_0(F)+r_k(p)$ and $L_{\kappa}^{\flat}(x_0)<n$,
\begin{equation}
\label{argmin-from-Fkdef}
\argmin_{i\geqslant 1}f_{\kappa}^{\flat}(x_0,i)\subseteq I_{\kappa}^{\flat}(x_0)\subseteq[k],\quad L_{\kappa}^{\flat}(x_0)=\min_{i\geqslant 1}f_{\kappa}^{\flat}(x_0,i).
\end{equation}

\subsection{Properties of stratification functions}
We begin by summarizing some simple properties of the functions we just defined, which will be of use in our upcoming stratification argument.

\begin{lemma}
For every $F\in\mathcal{F}_k(\mathbb{Z}/p^n\mathbb{Z})$ and $\kappa_0(F)+r_k(p)\leqslant\kappa\leqslant n$, let $L_{\kappa}(x_0)$, $I_{\kappa}^{\flat}(x_0)$, and $i_{\kappa}^{\flat}(x_0)$ be as in \eqref{Likappa-def}. Then:
\begin{enumerate}
\item\label{item1}
if $x_1\equiv x_0\pmod{p^{\kappa}}$, then
\begin{equation}
\label{local-constancy}
L_{\kappa}(x_1)=L_{\kappa}(x_0),\quad L_{\kappa}^{\flat}(x_1)=L_{\kappa}^{\flat}(x_0),\quad i_{\kappa}^{\flat}(x_1)=i_{\kappa}^{\flat}(x_0).
\end{equation}
\item\label{item2}
The sequences $(L_{\kappa}(x_0))$ and $(L_{\kappa}^{\flat}(x_0))$ are increasing in $\kappa$, and strictly increasing until the smallest value $\kappa_1$, $\kappa_1^{\flat}$ such that $L_{\kappa_1}(x_0)=L_{\kappa_1^{\flat}}^{\flat}(x_0)=n$. The sequence $(i_{\kappa}^{\flat}(x_0))$ is non-strictly decreasing for $\kappa<\kappa_1^{\flat}$.
\end{enumerate}
\end{lemma}

\begin{proof}
Writing $x_1=x_0+p^{\kappa}t$, from \eqref{diff-like} we obtain
\[ p^{\kappa i}\frac{F^{(i)}(x_0+p^{\kappa}t)}{i!}\equiv\sum_{j=0}^{n-i}\frac{F^{(i+j)}(x_0)}{i!j!}(p^{\kappa}t)^jp^{\kappa i}\equiv\sum_{j=0}^{n-i}\binom{i+j}i\frac{F^{(i+j)}(x_0)}{(i+j)!}p^{\kappa(i+j)}t^j\pmod{p^n}, \]
from which it follows that $L_{\kappa}(x_1)\geqslant L_{\kappa}(x_0)$; a symmetric argument establishes the reverse inequality. Similarly, we have
\[ ip^{\kappa i}\frac{F^{(i)}(x_0+p^{\kappa}t)}{i!}\equiv\sum_{j=0}^{n-i}\binom{i+j-1}{i-1}(i+j)\frac{F^{(i+j)}(x_0)}{(i+j)!}p^{\kappa(i+j)}t^j\pmod{p^n} \]
and consequently $L_{\kappa}^{\flat}(x_1)=L_{\kappa}^{\flat}(x_0)$. Of course, if $L_{\kappa}^{\flat}(x_0)=L_{\kappa}^{\flat}(x_1)=n$, then $i_{\kappa}^{\flat}(x_1)=i_{\kappa}^{\flat}(x_0)=k$; otherwise, when $L_{\kappa}^{\flat}(x_0)<n$, from the above we also conclude that for every $i>i_{\kappa}^{\flat}(x_0)$,
\[ f_{\kappa}^{\flat}(x_1,i)\geqslant\min_{i'>i_{\kappa}^{\flat}(x_0)}\min(n,f_{\kappa}^{\flat}(x_0,i'))>L_{\kappa}^{\flat}(x_0)=L_{\kappa}^{\flat}(x_1), \]
whence $i\not\in I_{\kappa}^{\flat}(x_1)$ and so $i_{\kappa}^{\flat}(x_1)\leqslant i_{\kappa}^{\flat}(x_0)$. A symmetric argument establishes the reverse inequality, thereby completing the proof of \eqref{local-constancy}. The claims in \eqref{item2} are clear, keeping in mind that $I_{\kappa}^{\flat}(x_0)\subseteq[k]$.
\end{proof}

In the next lemma, we summarize the direct consequences of the assumption that $F\in\mathcal{F}_k(\mathbb{Z}/p^n\mathbb{Z})$ for the lifting procedure.

\begin{lemma}
\label{Fkprop-lemma}
Let $F\in\mathcal{F}_k(\mathbb{Z}/p^n\mathbb{Z})$. For every $\kappa_0(F)+r_k(p)\leqslant\kappa\leqslant n$, let $L_{\kappa}(x_0)$, $I_{\kappa}^{\flat}(x_0)$, and $i_{\kappa}^{\flat}(x_0)$ be as in \eqref{Likappa-def}. Then there exists a polynomial $P_{\kappa}\in(\mathbb{Z}/p^n\mathbb{Z})[t]$ such that all of the following hold:
\begin{enumerate}
\item\label{cong-part1}
$F(x_0+p^{\kappa}t)\equiv F(x_0)+p^{L_{\kappa}}P_{\kappa}(t)\pmod{p^n}$;
\item\label{cong-part2}
for $L_{\kappa}^{\flat}<n$, we have $P_{\kappa}'(t)=p^{L_{\kappa}^{\flat}-L_{\kappa}}Q_{\kappa}(t)$ for some polynomial $Q_{\kappa}(t)$ whose reduction modulo $p$ is of degree $i_{\kappa}^{\flat}(x_0)-1$, with unit coefficients of $Q_{\kappa}(t)$ modulo $p$ precisely at terms $t^{i-1}$ with $i\in I_{\kappa}^{\flat}(x_0)$; and
\item\label{cong-part3} if $\kappa_0(F)+r_k(p)\leqslant\kappa_1\leqslant\kappa_2\leqslant n$ satisfy $L_{\kappa_2}^{\flat}<n$
and $|I_{\kappa_1}^{\flat}(x_0)\cap I_{\kappa_2}^{\flat}(x_0)|>1$, then $\kappa_1=\kappa_2$.
\end{enumerate}
\end{lemma}

\begin{proof}
We begin with the expansion
\begin{equation}
\label{expansion}
F(x_0+p^{\kappa}t)\equiv F(x_0)+\sum_{i=1}^n\frac{F^{(i)}(x_0)}{i!}(p^{\kappa}t)^i\pmod{p^n}.
\end{equation}
From our conditions, we have that
\[ \sum_{i=1}^n\frac{F^{(i)}(x_0)}{i!}(p^{\kappa}t)^i=p^{L_{\kappa}}P_{\kappa}(t),\quad P_{\kappa}(t):=\sum_{i=1}^n\Big(p^{-L_{\kappa}}\frac{F^{(i)}(x_0)}{i!}p^{\kappa i}\Big)t^i, \]
where further
\[ P_{\kappa}'(t)=p^{L_{\kappa}^{\flat}-L_{\kappa}}Q_{\kappa}(t),\quad Q_{\kappa}(t):=
\sum_{i=1}^n\Big(p^{-L_{\kappa}^{\flat}}i\frac{F^{(i)}(x_0)}{i!}p^{\kappa i}\Big)t^{i-1}. \]
Keeping in mind \eqref{argmin-from-Fkdef}, we see that the polynomial $Q(t)$ has unit coefficients precisely at $i\in I_{\kappa}^{\flat}(x_0)$, and the degree of $Q(t)$ modulo $p$ is precisely $i_{\kappa}^{\flat}(x_0)-1$. This establishes items~\eqref{cong-part1} and \eqref{cong-part2}.

As for item \eqref{cong-part3}, from
\[ F(x_0)+p^{L_{\kappa_1}}P_{\kappa_1}(p^{\kappa_2-\kappa_1}t)\equiv F(x_0+p^{\kappa_2}t)\equiv F(x_0)+p^{L_{\kappa_2}}P_{\kappa_2}(t)\pmod{p^n}, \]
which holds as a congruence of polynomials (that is, of their respective coefficients), it follows by differentiation that
\[ p^{L_{\kappa_1}^{\flat}+(\kappa_2-\kappa_1)}Q_{\kappa_1}(p^{\kappa_2-\kappa_1}t)\equiv p^{L_{\kappa_2}^{\flat}}Q_{\kappa_2}(t)\pmod{p^n}. \]
The polynomial on the right-hand side of this congruence has coefficients with $t^{i-1}$ exactly divisible by $p^{L_{\kappa_2}^{\flat}}$ at all $i\in I_{\kappa_2}^{\flat}(x_0)$, and in particular at the at least two distinct values of $i\in I_{\kappa_1}^{\flat}(x_0)\cap I_{\kappa_2}^{\flat}(x_0)$. In light of $L_{\kappa_2}^{\flat}<n$, the same must be true of those two coefficients on the left-hand side, but this is clearly only possible if $\kappa_1=\kappa_2$.
\end{proof}

\subsection{Lifting and stratification}
We now consider the family
\[ X_{\kappa}(\mathbb{Z}/p^n\mathbb{Z})=\{A\subseteq\mathbb{Z}/p^n\mathbb{Z}:A=A+p^{\kappa}\mathbb{Z}/p^n\mathbb{Z}\} \]
of $p^{\kappa}\mathbb{Z}/p^n\mathbb{Z}$-invariant subsets of $\mathbb{Z}/p^n\mathbb{Z}$, and sets $B_{\kappa}(x_0),B_{\kappa}^{\circ}(x_0)\in X_{\kappa}(\mathbb{Z}/p^n\mathbb{Z})$ given by
\[ B_{\kappa}(x_0)=\{x_0+p^{\kappa}t:t\in\mathbb{Z}/p^{n-\kappa}\mathbb{Z}\},\quad B_{\kappa}^{\circ}(x_0)=B_{\kappa}(x_0)\setminus B_{\kappa+1}(x_0). \]
Using \eqref{local-constancy}, we may denote by $L_{\kappa}(B_{\kappa}(x_0))$ and $i_{\kappa}^{\flat}(B_{\kappa}(x_0))$ the common values of $L_{\kappa}(x)$ and $i_{\kappa}^{\flat}(x)$ over these sets. We will also consider two families of these neighborhoods, as follows:
\begin{equation}
\label{classN-def}
\mathcal{N}_F=\left\{\bigsqcup_{t\in T}B_{\kappa+1}(x_0+tp^{\kappa}):
\begin{aligned}
&\kappa_0(F)\leqslant\kappa\leqslant n,\,\,x_0\in A_F,\,\, T\subseteq\mathbb{Z}/p\mathbb{Z},\\
&L_{\kappa}(x_0)\leqslant n-2-2\rho_k(p),\,\,\text{and}\\
&F(x_0+p^{\kappa}t)\equiv F(x_0)+p^{L_{\kappa}}P_{\kappa}(t)\bmod{p^n}\,\,(t\in T),\\
&P_{\kappa}\in(\mathbb{Z}/p^n\mathbb{Z})[t],\,\ord_pP_{\kappa}'(t)\leqslant\rho_k(p)\,\,(t\in T)
\end{aligned}\right\},
\end{equation}
\begin{equation}
\label{classF-def}
\mathcal{F}_F=\bigsqcup_{\kappa=\kappa_0(F)}^n\mathcal{F}_{F,\kappa},\quad \mathcal{F}_{F,\kappa}=\{B_{\kappa}(x_0),B_{\kappa}^{\circ}(x_0):L_{\kappa}(x_0)\geqslant n-1-2\rho_k(p)\},
\end{equation}
where $\rho_k(p)\in\mathbb{Z}_{\geqslant 0}$ is some fixed choice of constants (consistent throughout \eqref{classN-def} and \eqref{classF-def}) such that $\rho_k(p)\geqslant r_k(p)$ for all $p$ and $\rho_k(p)=0$ for all sufficiently large $p\geqslant p_0(k)$, and the condition $\ord_pP_{\kappa}'(t)\leqslant\rho_k(p)$ in the definition of $\mathcal{N}_F$ is required to be satisfied for every representative of $t\in T\subseteq\mathbb{Z}/p\mathbb{Z}$. The point of these families is that sums of $e(F(x)/p^n)$ over sets in families $\mathcal{N}_F$ and $\mathcal{F}_F$ are easy to understand; we will use this in section~\ref{proof-main-theorem}. For now, we first prove the following lemma.

\begin{lemma}
Let $F\in\mathcal{F}_k(\mathbb{Z}/p^n\mathbb{Z})$, $x_0\in A_F$, and $\kappa_0(F)+r_k(p)\leqslant\kappa\leqslant n$. Assume that $B=B_{\kappa}(x_0)\not\in\mathcal{F}_F$. Then:
\begin{enumerate}
\item\label{rec-step1} If, for any $x_0'\in B$, $|I_{\kappa}^{\flat}(x_0')|=1$, then
\[ \begin{cases}B=B_{\kappa}(x_0')=B_{\kappa+1}(x_0')\sqcup B_{\kappa}^{\circ}(x_0')\quad\text{with }B_{\kappa}^{\circ}(x_0')\in\mathcal{N}_F,&\text{if }i_{\kappa}^{\flat}(x_0')>1;\\
B=B_{\kappa}(x_0')\in\mathcal{N}_F,&\text{if }i_{\kappa}^{\flat}(x_0')=1.\end{cases} \]
\item\label{rec-step2} If $|I_{\kappa}^{\flat}(x)|>1$ for every $x\in B$, then there exist disjoint sets $X^{\flat}_{\kappa},X^{\sharp}_{\kappa}\subseteq B$, values $\kappa\leqslant\lambda_x\leqslant n$ ($x\in X^{\flat}_{\kappa}\sqcup X^{\sharp}_{\kappa}$), and a finite set $\mathcal{B}_{F,\kappa}\subseteq\mathcal{N}_F$ such that we have a decomposition
\begin{equation}
\label{decomp-target}
B=\bigsqcup_{x\in X^{\flat}_{\kappa}}B_{\lambda_x}(x)\sqcup\bigsqcup_{x\in X^{\sharp}_{\kappa}}B_{\lambda_x}(x)\sqcup\bigsqcup_{B''\in\mathcal{B}_{F,\kappa}}B''
\end{equation}
and such that
\[ |X^{\flat}_{\kappa}|,|X^{\sharp}_{\kappa}|=\OO\nolimits_k(1)\qquad\text{and}\qquad
\begin{alignedat}{3}
&B_{\lambda_x}(x)\not\in\mathcal{F}_F\text{ and }i_{\lambda_x}^{\flat}(B_{\lambda_x}(x))<i_{\kappa}^{\flat}(B)&&\quad (x\in X^{\flat}_{\kappa}),\\
&B_{\lambda_x}(x)\in\mathcal{F}_F&&\quad (x\in X^{\sharp}_{\kappa}).
\end{alignedat} \]
\end{enumerate}
\end{lemma}

\begin{proof}
We note that the condition $B_{\kappa}(x_0')\not\in\mathcal{F}_F$ implies that $L_{\kappa}^{\flat}(x_0')<n$. We first address item~\eqref{rec-step1}.
In light of $I_{\kappa}^{\flat}(x_0')=\{i_{\kappa}^{\flat}\}$, we have that $P_{\kappa}'(t)\equiv p^{\delta_{\kappa}}ut^{i_{\kappa}^{\flat}-1}\pmod{p^{\delta_{\kappa}+1}}$ for some unit $u\in(\mathbb{Z}/p^n\mathbb{Z})^{\times}$ and $0\leqslant\delta_k=L_k^{\flat}-L_k\leqslant r_k(p)$. From this, it is clear that $\ord_pP_{\kappa}'(t)\leqslant r_k(p)$ for $p\nmid t$, and that $\ord_pP_{\kappa}'(t)=\delta_k(p)\leqslant r_k(p)$ if for all $t$ when $i_{\kappa}^{\flat}=1$.
The claim \eqref{rec-step1} follows, keeping in mind that $L_{\kappa}(x_0')\leqslant n-2-2\rho_k(p)$ in light of $B_{\kappa}(x_0')\not\in\mathcal{F}_F$.

Now, we proceed to \eqref{rec-step2}. Recall from Lemma~\ref{Fkprop-lemma} that $F(x_0+p^{\kappa}t)\equiv F(x_0)+p^{L_{\kappa}}P_{\kappa}(t)\pmod{p^n}$. We have that $P_{\kappa}'(t)\equiv p^{\delta_{\kappa}}Q_{\kappa}(t)\pmod{p^{\delta_{\kappa}+1}}$ with $0\leqslant\delta_{\kappa}=L_{\kappa}^{\flat}-L_{\kappa}\leqslant r_k(p)$ and for some nonzero polynomial $Q_{\kappa}(t)$ whose degree modulo $p$ is at most $i_{\kappa}^{\flat}(x_0)\leqslant k$, and thus its set of roots modulo $p$, call it $W$, satisfies $|W|\leqslant k$.
We can now write
\begin{align*}
W_1&=\{t\in W:B_{\kappa+1}(x_0+tp^{\kappa})\in\mathcal{F}_F\}\\
W_2&=\{t\in W:i_{\kappa+1}^{\flat}(x_0+tp^{\kappa})<i_{\kappa}^{\flat}(x_0+tp^{\kappa})=i_{\kappa}^{\flat}(x_0)\}\setminus W_1,\\
W_3&=W\setminus (W_1\sqcup W_2),
\end{align*}
and
\begin{equation}
\label{decomp-start}
B=B_{\kappa}(x_0)=\Big(\bigsqcup_{t\in W_1}\sqcup\bigsqcup_{t\in W_2}\sqcup\bigsqcup_{t\in W_3}\sqcup\bigsqcup_{t\not\in W}\Big)B_{\kappa+1}(x_0+tp^{\kappa}).
\end{equation}

The equality \eqref{decomp-start} is the starting point for constructing the decomposition \eqref{decomp-target}. We start building this decomposition inductively, with the first step as follows:
\begin{itemize}
\item For $t\in W_1$, we set $\lambda_{x_0+tp^{\kappa}}=\kappa+1$ and add $x=x_0+tp^{\kappa}$ to the set $X_{\kappa}^{\sharp}$, noting that indeed $B_{\lambda_x}(x)\in\mathcal{F}_F$.
\item For $t\in W_2$, we set $\lambda_{x_0+tp^{\kappa}}=\kappa+1$ and add $x=x_0+tp^{\kappa}$ to the set $X_{\kappa}^{\flat}$, noting that indeed $i_{\lambda_{x_0+tp^{\kappa}}}^{\flat}(x_0+tp^{\kappa})<i_{\kappa}^{\flat}(x_0)$.
\item We note that $B''=\bigsqcup_{t\not\in W}B_{\kappa+1}(x_0+tp^{\kappa})\in\mathcal{N}_F$ and add $B''$ to the set $\mathcal{B}_{F,\kappa}$.
\item For $t\in W_3$, noting that $B_{\kappa+1}(x_0+tp^{\kappa})\not\in\mathcal{F}_F$ by construction, we now iteratively repeat the process for as long as it is possible to do so using the already proved item \eqref{rec-step1}. This produces a finite collection of sets in $\mathcal{N}_F$, which we add to $\mathcal{B}_{F,\kappa}$, and at most one remaining set $B_{\lambda_x'}(x)$ for some $x\equiv x_0+tp^{\kappa}\pmod{p^{\kappa+1}}$, $\lambda'_x\geqslant\kappa+1$, which is either:
\begin{itemize}
\item in $\mathcal{F}_F$, in which case we set $\lambda_x=\lambda_x'$ and add $x$ to the set $X_{\kappa}^{\sharp}$;
\item satisfies $i_{\lambda'_x}^{\flat}(x)<i_{\kappa}^{\flat}(x_0)$, in which case we set $\lambda_x=\lambda_x'$ and add $x$ to the set $X_{\kappa}^{\flat}$;
\item or satisfies $i_{\lambda'_x}^{\flat}(x)=\dots=i_{\kappa}^{\flat}(x_0)$ and $|I_{\lambda'_x}^{\flat}(x)|>1$ for every $x\in B_{\lambda_x}(x)$, in which case we add $x$ to a new set we call $X_{\kappa}^{(1)}$.
\end{itemize}
\end{itemize}
Putting everything together, this first step produces a decomposition of the form
\begin{equation}
\label{decomp-clean-step}
B=\bigsqcup_{x\in X_{\kappa}^{\sharp}\sqcup X_{\kappa}^{\flat}}B_{\lambda_x}(x)\sqcup\bigsqcup_{B''\in\mathcal{B}_{F,\kappa}}B''\sqcup\bigsqcup_{x\in X_{\kappa}^{(1)}}B_{\lambda_x'}(x),
\end{equation}
where $|X_{\kappa}^{\sharp}|,|X_{\kappa}^{\flat}|,|X_{\kappa}^{(1)}|=\OO_k(1)$, and
\[ |I_{\kappa}^{\flat}(x)|,|I_{\lambda_x'}^{\flat}(x)|>1,\,\,\lambda'_x>\kappa,\,\,B_{\lambda'_x}(x)\not\in\mathcal{F}_F\quad (x\in X_{\kappa}^{(1)}). \]

The process can now be repeated on each of the sets $B_{\lambda'_x}(x)$, yielding (after $r$ steps) collections $|X_{\kappa}^{\sharp}|,|X_{\kappa}^{\flat}|=\OO_k(1)^r$ and a collection $X_{\kappa}^{(r)}$ (with $|X_{\kappa}^{(r)}|=\OO_k(1)$ as well) such that $\lambda^{(r)}_x>\dots>\lambda^{(1)}_x=\lambda'_x>\lambda^{(0)}_x=\kappa$ and
\[ |I_{\lambda_x^{(r)}}^{\flat}(x)|,\dots,|I_{\lambda_x^{(1)}}^{\flat}(x)|,|I_{\lambda_x^{(0)}}^{\flat}(x)|>1,\,\,i_{\lambda_x^{(r)}}^{\flat}(x)=\dots=i_{\lambda_x^{(1)}}^{\flat}(x)=i_{\lambda_x^{(0)}}^{\flat}(x)\quad (x\in X_{\kappa}^{(r)}). \]
By Lemma~\ref{Fkprop-lemma}\eqref{cong-part3}, the sets $I_{\lambda_x^{(j)}}^{\flat}(x)\setminus\{i_{\lambda_x^{(j)}}^{\flat}(x)\}\subseteq\{1,\dots,i_{\kappa}^{\flat}(x)-1\}$ must be pairwise disjoint, whence $X_{\kappa}^{(r)}=\emptyset$ for $r\geqslant k$. This then yields the decomposition of a claimed form.
\end{proof}

By iterating the above lemma, we obtain the following statement.

\begin{lemma}
\label{split-domain-proposition}
Let $F\in\mathcal{F}_k(\mathbb{Z}/p^n\mathbb{Z})$, $x_0\in A_F$, and $\kappa_0(F)+r_k(p)\leqslant\kappa\leqslant n$. Then there exists finite collections $\mathcal{B}_F^{\sharp}\subseteq\mathcal{F}_F$ and $\mathcal{B}_F\subseteq\mathcal{N}_F$ such that 
\[ B_{\kappa}(x_0)=\bigsqcup_{B'\in\mathcal{B}_F^{\sharp}}B'\sqcup\bigsqcup_{B''\in\mathcal{B}_F}B'',\quad |\mathcal{B}_F^{\sharp}|=\OO\nolimits_k(1). \]
\end{lemma}

\section{Proof of Theorem~\ref{main-theorem}}
\label{proof-main-theorem}

In this section, we prove Theorem~\ref{main-theorem}. Using the results of section~\ref{preliminaries-section}, we will rewrite the sum $S(\bm{a};q)$ as the finite union of exponential sums with certain explicit phases $f_{\bm{\epsilon}}(x)$. We will use results of \S\ref{recurrence-sec} and section~\ref{alignment-section} to argue that these phases $f_{\bm{\epsilon}}(x)$ belong to specific classes $\mathcal{F}_k(\mathbb{Z}/p^n\mathbb{Z})$. Finally, we will use results of section~\ref{stratification-section} to stratify exponential sums with phases in $\mathcal{F}_k(\mathbb{Z}/p^n\mathbb{Z})$ to parts of the domain with substantial cancellation and those with essentially no cancellation, from which our bounds will follow.

\subsection{Estimation of exponential sums with phases in the class \texorpdfstring{$\mathcal{F}_k(\mathbb{Z}/p^n\mathbb{Z})$}{Fk(Z/pnZ)}}
In this subsection, we collect a few final ingredients and prove Proposition~\ref{Fk-final-estimate}, in which we estimate complete exponential sums with general phases in the class $\mathcal{F}_k(\mathbb{Z}/p^n\mathbb{Z})$.

Recall Definition~\ref{Fk-def} of the class $\mathcal{F}_k(\mathbb{Z}/p^n\mathbb{Z})$. In Proposition~\ref{split-domain-proposition} we split the domain of a function $F\in\mathcal{F}_k(\mathbb{Z}/p^n\mathbb{Z})$ as the finite union of domains in the classes $\mathcal{N}_F$ and $\mathcal{F}_{F,\lambda}$ defined in \eqref{classN-def} and \eqref{classF-def}. In the following two preparatory lemmata, we summarize the behavior of exponential sums restricted to such domains.

\begin{lemma}
\label{total-cxl-lemma}
Let $F\in\mathcal{F}_k(\mathbb{Z}/p^n\mathbb{Z})$ and $B\in\mathcal{N}_F$. Then
\[ \sum_{x\in B}e\Big(\frac{F(x)}{p^n}\Big)=0. \]
\end{lemma}

\begin{proof}
Using the definition of the class $\mathcal{N}_F$ as in \eqref{classN-def}, we have for $B=\bigsqcup_{t\in T}(x_0+tp^{\kappa})$ 
that
\begin{align*}
S_F(B):=\sum_{x\in B}e\Big(\frac{F(x)}{p^n}\Big)
&=e\Big(\frac{F(x_0)}{p^n}\Big)\sum_{t\in (T+p\mathbb{Z})/p^{n-\kappa}\mathbb{Z}}e\Big(\frac{P_{\kappa}(t)}{p^{n-L_{\kappa}}}\Big)\\
&=p^{L_{\kappa}-\kappa}e\Big(\frac{F(x_0)}{p^n}\Big)\sum_{t\in(T+p\mathbb{Z})/p^{n-L_{\kappa}}\mathbb{Z})}e\Big(\frac{P_{\kappa}(t)}{p^{n-L_{\kappa}}}\Big).
\end{align*}
Denoting $\tilde{\kappa}=\lfloor (n-L_{\kappa})/2\rfloor$ and using the stationary phase Lemma~\ref{statphase-lemma} with $\kappa\mapsto n-L_{\kappa}-\tilde{\kappa}$, we have that
\[ S_F(B)=p^{L_{\kappa}-\kappa+\tilde{\kappa}}e\Big(\frac{F(x_0)}{p^n}\Big)\sum_{\substack{t\in(T+p\mathbb{Z})/p^{n-L_{\kappa}-\tilde{\kappa}}\mathbb{Z}\\P_{\kappa}'(t)\equiv 0\bmod{p^{\tilde{\kappa}}}}}e\Big(\frac{P_{\kappa}(t)}{p^{n-L_{\kappa}}}\Big), \]
keeping in mind that the defining conditions of the class $\mathcal{N}_F$ guarantee that $n-L_{\kappa}-\tilde{\kappa}\geqslant\max(1,(n-L_{\kappa})/2)$. Since $n-L_{\kappa}\geqslant 2+2\rho_k(p)$, we have that $\tilde{\kappa}\geqslant 1+\rho_k(p)$. On the other hand, the definition of the class $\mathcal{N}_F$ requires that $\ord_pP_{\kappa}'(t)\leqslant\rho_k(p)$ for all $t\in T+p\mathbb{Z}$. Therefore, the sum over $t$ is empty, and $S_F(B)=0$.
\end{proof}

Over domains $B\in\mathcal{N}_F$, Lemma~\ref{total-cxl-lemma} shows that the phase in the exponential sum $\sum_{x\in B} e(F(x)/p^n)$ exhibits sufficient oscillation (over small $p$-adic neighborhoods) that the entire sum cancels out. Contrasting this is the situation when $B\in\mathcal{F}_{F,\kappa}$: in light of the power series expansion \eqref{expansion}, the phase is constant modulo $p^{L_{\kappa}}$ with $L_{\kappa}\geqslant n-1-2C$, so that there is no reason to expect significant cancellation beyond some accidental coincidence. At that point, we thus bound the sum trivially and estimate its size in the following lemma.

\begin{lemma}
\label{estimate-noncxl}
Let $F\in\mathcal{F}_k(\mathbb{Z}/p^n\mathbb{Z})$, $\kappa_0(F)\leqslant\kappa\leqslant n$, and $B\in\mathcal{F}_{F,\kappa}$. Then
\[ \bigg|\sum_{x\in B}e\Big(\frac{F(x)}{p^n}\Big)\bigg|\leqslant |B|\asymp p^{n-\kappa}\ll p^{n-\lceil\frac{n-\Delta_k(F)-1-2\rho_k(p)}{k}\rceil}\ll_kp^{n-\frac{n-\Delta_k(F)-1}{k}}. \]
\end{lemma}

\begin{proof}
From
\[ n-1-2\rho_k(p)\leqslant L_{\kappa}(x_0)=\min_{1\leqslant i\leqslant k}\Big[\min\Big(n,\ord\nolimits_p\frac{F^{(i)}(x_0)}{i!}+\kappa i\Big)\Big]\leqslant\Delta_k(F)+k\kappa \]
it follows that
\[ \kappa\geqslant\Big\lceil\frac{n-\Delta_k(F)-1-2\rho_k(p)}{k}\Big\rceil. \]
The lemma is immediate from this.
\end{proof}

Combining the results of Lemmata~\ref{split-domain-proposition}, \ref{total-cxl-lemma} and \ref{estimate-noncxl}, we will obtain the following proposition, which is our main estimate for complete exponential sums with phases in $\mathcal{F}_k(\mathbb{Z}/p^n\mathbb{Z})$.

\begin{proposition}
\label{Fk-final-estimate}
For every odd prime $p$, $q=p^n$, $F\in\mathcal{F}_k(\mathbb{Z}/p^n\mathbb{Z})$, and any $p^{\kappa_0(F)}\mathbb{Z}/p^n\mathbb{Z}$-invariant weight $g:A_F\to\mathbb{C}$ with $|g|\ll_k1$,
\[ \sum_{x\in A_F}g(x)e\Big(\frac{F(x)}{p^n}\Big)\ll_kp^{n-\frac{n-\Delta_k(F)-1}{k}+1}, \]
where $\Delta_k(F)\leqslant n$ is as in Definition~\ref{Fk-def}.
\end{proposition}

\begin{proof}
Denote $\kappa_1(F)=\kappa_0(F)+r_k(p)$. We begin by writing
\[ \sum_{x\in A_F}g(x)e\Big(\frac{F(x)}{p^n}\Big)=\sum_{[x_0]\in A_F/p^{\kappa_1(F)}\mathbb{Z}}g(x_0)\sum_{x\in B_{\kappa_1(F)}(x_0)}e\Big(\frac{F(x)}{p^n}\Big). \]
Applying Lemma~\ref{split-domain-proposition} to $B_{\kappa_1(F)}(x_0)$, we obtain finite collections $\mathcal{B}_F^{\sharp}(x_0)\subseteq\mathcal{F}_F$ and $\mathcal{B}_F(x_0)\subseteq\mathcal{N}_F$ such that $|\mathcal{B}_F^{\sharp}(x_0)|=\OO_k(1)$ and
\[ \sum_{x\in B_{\kappa_1(F)}(x_0)}e\Big(\frac{F(x)}{p^n}\Big)=\sum_{B'\in\mathcal{B}_F^{\sharp}(x_0)}\sum_{x\in B'}e\Big(\frac{F(x)}{p^n}\Big)+\sum_{B''\in\mathcal{B}_F(x_0)}\sum_{x\in B''}e\Big(\frac{F(x)}{p^n}\Big). \]
Applying Lemma~\ref{estimate-noncxl} to each of the sums over $B'\in\mathcal{B}_F^{\sharp}(x_0)$ and Lemma~\ref{total-cxl-lemma} to each of the sums over $B''\in\mathcal{B}_F(x_0)$ and putting everything together, we obtain
\[ \sum_{x\in A_F}g(x)e\Big(\frac{F(x)}{p^n}\Big)\ll_k\sum_{[x_0]\in A_F/p^{\kappa_1(F)}\mathbb{Z}}|\mathcal{B}_F^{\sharp}(x_0)|\cdot p^{n-\frac{n-\Delta_k(F)-1}{k}}\ll_k p^{n-\frac{n-\Delta_k(F)-1}{k}+1}, \]
as announced.
\end{proof}

\subsection{Recurrence properties of the Kloosterman phase}
We will need the following statement combining the reasoning from Lemma~\ref{recurrence-lemma}, \S\ref{class-subsection}, and Lemma~\ref{alignment-mod-p-lemma}.

\begin{lemma}
\label{khalf-lemma}
Let $k\in\mathbb{N}$. For any two vectors $\bm{a}=(a_i)\in(\mathbb{Z}/p^n\mathbb{Z})^k$ and $\bm{\epsilon}=(\epsilon_i)\in\{\pm 1\}^k$ and for every $x\in X_{\bm{a}}$ as in \eqref{Xa-domain}, let
\begin{equation}
\label{sigmam-xae-def}
\sigma_m(x;\bm{a},\bm{\epsilon})=\sum_{i=1}^k\epsilon_i(x+a_i)_{1/2}^{1-2m}.
\end{equation}
Then there exist constants $\rho_k(p)\in\mathbb{Z}_{\geqslant 0}$ (over all odd primes $p$) such that $\rho_k(p)=0$ for all sufficiently large $p\geqslant p_0(k)$ such that if \begin{equation}
\label{khalf-lemma-given}
p^{r+\rho_k(p)}\mid\sigma_m(x;\bm{a},\bm{\epsilon})\quad (1\leqslant m\leqslant\lceil k/2\rceil),
\end{equation}
then
\begin{equation}
\label{khalf-lemma-conclusion}
p^{r+\rho_k(p)+1}\mid p^{(1+2\rho_k(p))(m-\lceil k/2\rceil)}\frac{\sigma_m(x;\bm{a},\bm{\epsilon})}{m!}\quad (m>\lceil k/2\rceil).
\end{equation}
\end{lemma}

\begin{proof}
As in the proof of Lemma~\ref{alignment-mod-p-lemma}, denoting $\bm{y}=(y_i)$ with $y_i=\epsilon_i(x+a_i)_{1/2}^{-1}$, we have that $\sigma_m(x;\bm{a},\bm{\epsilon})=\sigma_{2m-1}(\bm{y})$, and the given conditions \eqref{khalf-lemma-given} imply that
\[ \sigma_m(\bm{y})\equiv\sigma_m(-\bm{y})\pmod{p^{r+\rho_k(p)}}\quad (1\leqslant m\leqslant 2\lceil k/2\rceil). \]
Using Newton's formulae \eqref{non-recursive} as in the proof of Lemma~\ref{alignment-mod-p-lemma}, this further implies that
\[ e_m(\bm{y})\equiv e_m(-\bm{y})\pmod{p^r}\quad(1\leqslant m\leqslant k). \]

Now, we continue the argument as in \S\ref{class-subsection}. Using the recurrence \eqref{newton-recurrence}, we conclude that
\[ \sigma_m(\bm{y})\equiv\sigma_m(-\bm{y})\pmod{p^r},\quad\text{whence}\quad p^r\mid\sigma_m(x;\bm{a},\bm{\epsilon})\quad(m\in\mathbb{N}). \]
Coupled with the elementary estimate $\ord_p(m!)<m/(p-1)\leqslant m/2\leqslant (m-\lceil k/2\rceil)$ for $m\geqslant 2\lceil k/2\rceil$, this in turn settles the claim \eqref{khalf-lemma-conclusion} (with $\rho_k(p)=0$) whenever $m>k$ or $p>k$. It remains to handle the case when $\lceil k/2\rceil<m\leqslant k$ and $3\leqslant p\leqslant k$, which as in \S\ref{class-subsection} can be done by adjusting the finitely many values of $\rho_k(p)$, specifically,  by requiring additionally that
\[ \rho_k(p)\geqslant\max_{\lceil k/2\rceil<m\leqslant k}\frac{\ord_pm!}{m-\lceil k/2\rceil}. \qedhere \]
\end{proof}

A fairly direct consequence of the above is the following statement.

\begin{lemma}
\label{khalf-lemma2}
For $k\in\mathbb{N}$, $\bm{a}\in(\mathbb{Z}/p^n\mathbb{Z})^k$, $\bm{\epsilon}\in\{\pm 1\}^k$, and $x\in X_{\bm{a}}$, let $\sigma_m(x;\bm{a},\bm{\epsilon})$ be as in \eqref{sigmam-xae-def}, and let
\[ \Delta_k(x;\bm{a},\bm{\epsilon})=\min_{1\leqslant m\leqslant\lceil k/2\rceil}\ord_p\Big((1/2)_m\frac{\sigma_m(x;\bm{a},\bm{\epsilon})}{m!}\Big). \]
Then there exist constants $\rho_k(p)\in\mathbb{Z}_{\geqslant 0}$ (over all odd primes $p$) such that $\rho_k(p)=0$ for all sufficiently large $p\geqslant p_0(k)$ such that
\[ p^{\Delta_k(x;\bm{a},\bm{\epsilon})+1}\mid p^{(1+\rho_k(p))(m-\lceil k/2\rceil)^{+}}(1/2)_m\frac{\sigma_m(x;\bm{a},\bm{\epsilon})}{m!}\quad (m>\lceil k/2\rceil). \]
\end{lemma}

\begin{proof}
For clarity, re-label the values $\rho_k(p)$ and $p_0(k)$ guaranteed by Lemma~\ref{khalf-lemma} as $\rho^{\circ}_k(p)$ and $p^{\circ}_0(k)$, respectively.
Noting that trivially
\[ \Delta_k(x;\bm{a},\bm{\epsilon})\leqslant\ord_p(1/2)_{\lceil k/2\rceil}+\Delta_k^{\circ}(x;\bm{a},\bm{\epsilon}),\quad
\Delta_k^{\circ}(x;\bm{a},\bm{\epsilon}):=\min_{1\leqslant m\leqslant\lceil k/2\rceil}\sigma_m(x;\bm{a},\bm{\epsilon}), \]
the statement of Lemma~\ref{khalf-lemma2} follows directly from Lemma~\ref{khalf-lemma} for $\Delta_k^{\circ}(x;\bm{a},\bm{\epsilon})\geqslant\rho_k^{\circ}(p)$ (with $\rho_k(p)=2\rho_k^{\circ}(p)$). In particular, this handles all $p\geqslant p^{\circ}_0(k)$. It remains to treat the cases $3\leqslant p<p^{\circ}_0(k)$ and $\Delta_k^{\circ}(x;\bm{a},\bm{\epsilon})\leqslant\rho_k^{\circ}(p)-1$, when we can simply take
\[ \rho_k(p)=\max_{\lceil k/2\rceil<m\leqslant2\lceil k/2\rceil}\Big\lfloor\frac{\rho_k^{\circ}(p)+\ord_p(m!)+1}{m-\lceil k/2\rceil}\Big\rfloor. \qedhere \]
\end{proof}

\subsection{Proof of Theorem~\ref{main-theorem}}
\label{main-theorem-proof-subsection}

We are now finally ready for the proof of Theorem~\ref{main-theorem}.

\begin{proof}[Proof of Theorem~\ref{main-theorem}]
Clearly we may assume that $k\geqslant 2$, since the theorem is otherwise vacuous.
By using the classical evaluation of the Kloosterman sums in Lemma~\ref{kloost-eval} and expanding, we may write
\begin{equation}
\label{Sfe-eq}
\begin{gathered}
S(\bm{a};p^n)=\sum_{\bme\in\{\pm 1\}^k}S(f_{\bm{\epsilon}};g_{\bm{\epsilon}}), \qquad S(f_{\bm{\epsilon}};g_{\bm{\epsilon}})=\sum_{x\in X_{\bm{a}}}g_{\bm{\epsilon}}(x)e\Big(\frac{f_{\bm{\epsilon}}(x)}{p^n}\Big),\\
f_{\bm{\epsilon}}(x)=\sum_{i=1}^n 2\epsilon_i(x+a_i)_{1/2},\,\, g_{\bm{\epsilon}}(x)=\big({\textstyle\prod\nolimits_{i=1}^n\epsilon_i}\big)^{\frac{1+(p^n/4)}2}\Big(\frac{\prod_{i=1}^n(x+a_i)_{1/2}}{p^n}\Big)\quad (x\in X_{\bma}),
\end{gathered}
\end{equation}
where as in \S\ref{class-subsection} we write for short
\begin{equation}
\label{Xa-domain}
X_{\bma}=\{x\in\mathbb{Z}/p^n\mathbb{Z}:x+a_i\in(\mathbb{Z}/p^n\mathbb{Z})^{\times 2}\,(1\leqslant i\leqslant k)\}.
\end{equation}
We have already established in \S\ref{class-subsection} that the function $f_{\bm{\epsilon}}:X_{\bm{a}}\to\mathbb{Z}/p^n\mathbb{Z}$ satisfies item \eqref{Fk-def-item1} of Definition~\ref{Fk-def} for any $\kappa_0\geqslant 1$. On the other hand, in light of
\[ f_{\bm{\epsilon}}^{(i)}(x)=2(1/2)_i\sigma_{i}(x;\bm{a},\bm{\epsilon}) \]
in the notation of \eqref{sigmam-xae-def}, Lemma~\ref{khalf-lemma2} then shows that $f_{\bm{\epsilon}}$ also satisfies item \eqref{Fk-def-item2} of Definition~\ref{Fk-def} with $(1+\rho_k(p),\lceil k/2\rceil)$ in place of $(\kappa_0,k)$. This proves that the phase $f_{\bm{\epsilon}}$ belongs to the class $\mathcal{F}_{\lceil k/2\rceil}(\mathbb{Z}/p^n\mathbb{Z})$ of Definition~\ref{Fk-def}:
\[ f_{\bm{\epsilon}}\in \mathcal{F}_{\lceil k/2\rceil}(\mathbb{Z}/p^n\mathbb{Z}). \]

Next, for an arbitrary $x\in A_{f_{\bm{\epsilon}}}=X_{\bm{a}}$, we turn our attention to the quantity
\[ \Delta_{\lceil k/2\rceil}(f_{\bm{\epsilon}};x)=\min_{1\leqslant i\leqslant \lceil k/2\rceil}\ord\nolimits_p\frac{2(1/2)_i\sigma_i(x;\bm{a},\bm{\epsilon})}{i!}, \]
which satisfies
\begin{equation}
\label{Delta-x-estimate}
r\leqslant\Delta_{\lceil k/2\rceil}(f_{\bm{\epsilon}};x)\leqslant r+\rho_k(p), \quad r:=\min_{1\leqslant i\leqslant \lceil k/2\rceil}\ord\nolimits_p\sigma_i(x;\bm{a},\bm{\epsilon}),
\end{equation}
and $\rho_k(p)\geqslant 0=\max_{1\leqslant i\leqslant\lceil k/2\rceil}\ord_p\binom{1/2}{i}\geqslant 0$ are constants such that $\rho_k(p)=0$ for all sufficiently large $p>2\lceil k/2\rceil$. Recalling \eqref{sigmam-xae-def} and \eqref{alignment-section-system}, the condition that $p^r\mid\sigma_i(x;\bm{a},\bm{\epsilon})$ for all $1\leqslant i\leqslant\lceil k/2\rceil$ is precisely equivalent to the system of congruences $\Sigma_{\bm{\epsilon}}(k,\lceil k/2\rceil;\bm{a};p^r)$ considered in section~\ref{alignment-section}.
Moreover, if $2\mid k$ and the splitting conditions~\eqref{kt-def-splitting} are satisfied, then Lemma~\ref{recurrence-lemma-plus-p} applies and shows that $\ord_p\sigma_i(x;\bm{a},\bm{\epsilon})\geqslant r-\rho^{\ast}_k(p)$ for all $i\in\mathbb{N}$ for some $\rho_k^{\ast}(p)\in\mathbb{Z}_{\geqslant 0}$ such that $\rho_k^{\ast}(p)=0$ for sufficiently large $p\geqslant p_0^{\ast}(k)$, so that in this case the system $\Sigma_{\bm{\epsilon}}(k,\ell;\bm{a};p^{r-\rho_k(p)})$ is also satisfied for every $\ell\in\mathbb{N}$. The consistency of these systems (since $x\in A_{f_{\bm{\epsilon}}}$ satisfies it) implies by Lemma~\ref{alignment-main} that
\[ r-\rho_k^{\ast}(p)\leqslant\Delta_{\bm{\epsilon}}^{\ast}(\bm{a})+\rho_k'(p), \]
where
\[ \Delta_{\bm{\epsilon}}^{\ast}(\bm{a})=\begin{cases} 0,&\begin{aligned}&2\nmid k\text{ or the splitting conditions}\\&\text{\ \ \eqref{kt-def-splitting} are not satisfied},\end{aligned}\\\min\limits_{t\in T}\Delta(\bm{a}_t^{I_t},\bm{a}_t^{I_t'}),&\text{otherwise}, \end{cases} \]
for some constants $\rho_k'(p)\in\mathbb{Z}_{\geqslant 0}$ such that $\rho_k'(p)=0$ for all sufficiently large $p\geqslant p_0'(k)$, and the notations $[k_t]=I_t\sqcup I'_t$ and $\Delta(\bm{a}_t^{I_t},\bm{a}_t^{I_t'})$ are as in \eqref{kt-def-splitting} and \eqref{Delta-ta-def}. Using this estimate in the bound \eqref{Delta-x-estimate}, which holds for every $x\in A_{f_{\bm{\epsilon}}}$, we conclude that
\[ \Delta_{\lceil k/2\rceil}(f_{\bm{\epsilon}})=\max_{x\in A_{f_{\bm{\epsilon}}}}\Delta_{\lceil k/2\rceil}(f_{\bm{\epsilon}};x)\leqslant\Delta^{\ast}_{\bm{\epsilon}}(\bm{a})+\rho_k''(p), \]
where we set $\rho_k''(p)=\rho_k(p)+\rho_k'(p)+\rho_k^{\ast}(p)\in\mathbb{Z}_{\geqslant 0}$, and $\rho_k''(p)=0$ for all sufficiently large $p\geqslant\max(2\lceil k/2\rceil,p_0'(k),p_0^{\ast}(k))$.

The stage is now set for the application of the crucial Proposition~\ref{Fk-final-estimate}: since $f_{\bm{\epsilon}}\in\mathcal{F}_{\lceil k/2\rceil}(\mathbb{Z}/p^n\mathbb{Z})$ and $g_{\bm{\epsilon}}$ is $p^{\kappa_0(f_{\bm{\epsilon}})}\mathbb{Z}/p^n\mathbb{Z}$-invariant, this yields
\[ S(f_{\bm{\epsilon}};g_{\bm{\epsilon}})=\sum_{x\in X_{\bm{a}}}g_{\bm{\epsilon}}(x)e\Big(\frac{f_{\bm{\epsilon}}(x)}{p^n}\Big)\ll_kp^{n-\frac{n-\Delta_{\lceil k/2\rceil}(f_{\bm{\epsilon}})-1}{\lceil k/2\rceil}+1}\ll_kp^{n-\frac{n-\Delta^{\ast}_{\bm{\epsilon}}(\bm{a})-1}{\lceil k/2\rceil}+1}. \]
Summing over all $\bm{\epsilon}\in\{\pm 1\}^k$, we conclude that
\[ S(\bm{a};q)=\sum_{\bm{\epsilon}\in\{\pm 1\}^k}S(f_{\bm{\epsilon}};g_{\bm{\epsilon}})\ll_k\max_{\bm{\epsilon}\in\{\pm 1\}^k}p^{n-\frac{n-\Delta^{\ast}_{\bm{\epsilon}}(\bm{a})-1}{\lceil k/2\rceil}+1}
=p^{n-\frac{n-\Delta^{\ast}(\bm{a})-1}{\lceil k/2\rceil}+1}, \]
where $\Delta^{\ast}(\bm{a})=\max_{\bm{\epsilon}\in\{\pm 1\}^k}\Delta^{\ast}_{\bm{\epsilon}}(\bm{a})$ is as in Definition~\ref{Deltaast-maindef}. The proof of Theorem~\ref{main-theorem} is complete.
\end{proof}

\begin{example}
\label{mu-example}
We now consider the case \eqref{mu-example-eq} of the sum
\[ S_k(\bm{a};q)=\sumast_{x\bmod q}\Kl_2(x+\mu_kp^m;q)\Kl_2(x+\mu_k^2p^m;q)\cdots\Kl_2(x+\mu_k^kp^m;q)\Kl_2(x,q)^k, \]
with $\bm{a}=(\mu_kp^m,\mu_k^2p^m,\dots,\mu_k^{k-1}p^m,p^m,0,0,\dots,0,0)\in(\mathbb{Z}/p^n\mathbb{Z})^{2k}$, for which we claim full size $S_k(\bm{a};q)\asymp q$. Expanding using the Kloosterman sum evaluation in Lemma~\ref{kloost-eval} as in the proof of Theorem~\ref{main-theorem}, we can write
\[ S_k(\bm{a};q)=S(f_{\bm{\epsilon}_{\pm}};g_{\bm{\epsilon}_{\pm}})+S(f_{\bm{\epsilon}_{\mp}};g_{\bm{\epsilon}_{\mp}})+\sum_{\substack{\bm{\epsilon}\in\{\pm 1\}^{2k}\\\bm{\epsilon}\neq\bm{\epsilon}_{\pm},\bm{\epsilon}_{\mp}}}S(f_{\bm{\epsilon}};g_{\bm{\epsilon}}), \]
where $\bm{\epsilon}_{\pm}=((1)_k,(-1)_k)$ and $\bm{\epsilon}_{\mp}=((-1)_k,(1)_k)$ with $(\epsilon)_k=(\epsilon,\epsilon,\dots,\epsilon)\in\{\pm 1\}^k$, and $f_{\bm{\epsilon}}$, $g_{\bm{\epsilon}}$ are as in \eqref{Sfe-eq}, with $A_{f_{\bm{\epsilon}}}=(\mathbb{Z}/q\mathbb{Z})^{\times 2}$. Now, using the expansion \eqref{sqroot-diff}, we compute that
\begin{align*}
f_{\bm{\epsilon}_{\pm}}(x)
&=\sum_{i=1}^k2(x+\mu_k^ip^m)_{1/2}-2k(x)_{1/2}\\
&\equiv\sum_{j=0}^{k-1}2\binom{1/2}{j}x_{1/2}^{1-2j}p^{jm}\sum_{i=1}^k\mu_k^{ij}-2k(x)_{1/2}\equiv 0\pmod{p^n},
\end{align*}
since the inner sum over $i$ vanishes unless $j=0$. Similarly, we find that $f_{\bm{\epsilon}_{\mp}}(x)=0$, and we also see that $g_{\bm{\epsilon}_{\pm}}(x)=g_{\bm{\epsilon}_{\mp}}(x)=(-1)^{k(1+(p^n/4))/2}=c$ is a non-zero constant. Thence
\[ S(f_{\bm{\epsilon}_{\pm}};g_{f_{\bm{\epsilon}_{\pm}}})=S(f_{\bm{\epsilon}_{\mp}};g_{f_{\bm{\epsilon}_{\mp}}})=\tfrac12c(1-1/p)\cdot q. \]
On the other hand, for $\bm\epsilon\not\in\{\bm{\epsilon}_{\pm},\bm{\epsilon}_{\mp}\}$, consider the sets $I^{+}$ and $I^{-}$ given by $I^{\epsilon}=\{1\leqslant i\leqslant 2k:\epsilon_i=\epsilon\}$. Unless $|I^{+}|=|I^{-}|=k$, the splitting conditions \eqref{kt-def-splitting} are not satisfied, so $\Delta_{\bm{\epsilon}}^{\ast}(\bm{a})=0$ and the proof of Theorem~\ref{main-theorem} shows that
\[ S(f_{\bm{\epsilon}};g_{\bm{\epsilon}})\ll_kp^{n-\frac{n-1}{k}+1}. \]
Otherwise, when $|I^{+}|=|I^{-}|=k$ and $\bm{\epsilon}\not\in\{\bm{\epsilon}_{\pm},\bm{\epsilon}_{\mp}\}$, we have that
\[ \prod_{i\in I^{+}\cap[k]}(X-\mu_k^i)X^{|I^{+}\setminus[k]|}\neq\prod_{i\in I^{-}\cap[k]}(X-\mu_k^i)X^{|I^{-}\setminus[k]|}\quad\text{in }(\mathbb{Z}/p\mathbb{Z})[X]. \]
Therefore, there exists an $0<j\leqslant\max(|I^{+}\cap[k]|,|I^{-}\cap[k]|)<k$ such that
\[ \ord_p\Big(\sum_{i\in I^{+}\cap[k]}\mu_k^{ij}-\sum_{i\in I^{-}\cap[k]}\mu_k^{ij}\Big)\leqslant\rho_k(p), \]
for some constant $\rho_k(p)\geqslant 0$ such that $\rho_k(p)=0$ for all $p\geqslant p_0(k)$. But then the corresponding power sums of $\bm{a}^{\epsilon}=\{a_i:\epsilon_i=\epsilon\}$ satisfy
\[ \Delta_j(\bm{a}^{+},\bm{a}^{-})\leqslant mj+\rho_k(p), \]
whence $\Delta(\bm{a}^{+},\bm{a}^{-})\leqslant n-m+\rho_k(p)$. The proof of Theorem~\ref{main-theorem} shows that
\[ S(f_{\bm{\epsilon}};g_{\bm{\epsilon}})\ll_kp^{n(1-1/k^2)+1+1/k} \]
and so, as claimed,
\[ S_k(\bm{a};q)=c(1-1/p)q+\OO\nolimits_{k,p}(q^{1-1/k^2})\asymp_{k,p}q. \]
\end{example}

\bibliographystyle{amsalpha}
\bibliography{SumProKloost-ref}

\end{document}